\documentclass[11pt]{article}
\usepackage{fullpage}
\usepackage[hyperindex,breaklinks,colorlinks,citecolor=blue]{hyperref}\usepackage{amsmath, amssymb, amsthm, graphicx}
\usepackage{subcaption}

\usepackage{accents}
\usepackage{mathrsfs}
\usepackage{enumerate}
\theoremstyle{plain}
\newtheorem{thm}{Theorem}[section]
\newtheorem{prop}[thm]{Proposition}
\newtheorem{lem}[thm]{Lemma}
\newtheorem{cor}[thm]{Corollary}
\newtheorem{theorem*}{Theorem}

\theoremstyle{definition}
\newtheorem{defn}[thm]{Definition}

\theoremstyle{remark}
\newtheorem{claim}[thm]{Claim}
\newtheorem{rem}[thm]{Remark}
\newtheorem{fact}[thm]{Fact}
\newtheorem{notation}[thm]{Notation}
\newtheorem{example}[thm]{Example}
\newtheorem{question}{Question}

\newcommand{\dom}{\mathrm{dom}}

\newcommand{\V}{\tau_{\mathcal{V}}}
\newcommand{\upV}{\tau_{up\mathcal{V}}}
\newcommand{\upVV}{\upV^{2}}
\newcommand{\lowV}{\tau_{low\mathcal{V}}}

\newcommand{\cyl}[1]{\mathrm{Cyl}(#1)}

\newcommand{\M}[1]{\mathrm{mesh}(#1)}

\newcommand{\cH}{\check{H}}
\newcommand{\VV}{\V^{2}}
\renewcommand{\int}[2]{\mathrm{int}_{#1}(#2)}
\newcommand{\diam}[1]{\mathrm{diam}(#1)}

\newcommand{\ur}[1]{\text{\ensuremath{\uparrow}}#1}
\newcommand{\Ne}[2]{\mathcal{N}_{#1}(#2)}
\renewcommand{\d}[3]{\text{d}_{#1}(#2,#3)}

\newcommand{\E}{\mathsf{E}}
\newcommand{\I}{\mathsf{I}}
\renewcommand{\H}{\mathsf{H}}

\newcommand{\Max}{\mathsf{Max}}

\newcommand{\con}{\mathcal{C}}
\newcommand{\inj}{\mathcal{I}}
\newcommand{\D}{\mathcal{D}}
\newcommand{\F}{\mathcal{F}}
\newcommand{\K}{\mathcal{K}}
\newcommand{\J}{\mathcal{J}}
\renewcommand{\L}{\mathcal{L}}
\newcommand{\R}{\mathbb{R}}
\newcommand{\N}{\mathbb{N}}
\renewcommand{\P}{\mathcal{P}}
\newcommand{\Q}{\mathcal{Q}}
\newcommand{\U}{\mathcal{U}}
\newcommand{\Z}{\mathbb{Z}}

\newcommand{\cB}{\overline{B}}
\newcommand{\B}{\mathbb{B}}
\renewcommand{\SS}{\mathbb{S}}

\newcommand{\SCT}{\mathsf{SCT}}
\newcommand{\C}{\mathsf{C}}

\newcommand{\cone}[1]{\mathrm{Cone}(#1)}
\newcommand{\id}{\mathrm{id}}

\DeclareMathAccent{\wtilde}{\mathord}{largesymbols}{"65}
\newcommand{\SSigma}{\underaccent{\wtilde}{{\boldsymbol{\Sigma}}}}

\newcommand{\ceil}[1]{\lceil #1 \rceil}

\title{Strong computable type\footnote{This version is an update from July 2023, in which Section \ref{sec_infinite} was added.}}

\author{Djamel Eddine Amir and Mathieu Hoyrup\\ \\
{\small Universit\'e de Lorraine, CNRS, Inria, LORIA, F-54000 Nancy, France}\\
{\small\texttt{djamel-eddine.amir@loria.fr, mathieu.hoyrup@inria.fr}}}

\date{}
\begin{document}
\maketitle

\begin{abstract}
A compact set has computable type if any homeomorphic copy of the set which is semicomputable is actually computable. Miller proved that finite-dimensional spheres have computable type, Iljazovi\'c and other authors established the property for many other sets, such as manifolds. In this article we propose a theoretical study of the notion of computable type, in order to improve our general understanding of this notion and to provide tools to prove or disprove this property.

We first show that the definitions of computable type that were distinguished in the literature, involving metric spaces and Hausdorff spaces respectively, are actually equivalent. We argue that the stronger, relativized version of computable type, is better behaved and prone to topological analysis. We obtain characterizations of strong computable type, related to the descriptive complexity of topological invariants, as well as purely topological criteria. We study two families of topological invariants of low descriptive complexity, expressing the extensibility and the null-homotopy of continuous functions. We apply the theory to revisit previous results and obtain new ones.
\end{abstract}


\section{Introduction}

Computable analysis provides several notions of computability for compact subsets of Euclidean spaces and more general topological spaces. The most important ones are the notions of \emph{computable} and \emph{semicomputable} set. Intuitively, a subset of the Euclidean plane is computable if there exists a program that can draw the set on a screen with arbitrary resolution; it is semicomputable if there exists a program that can reject points that are outside the set. A famous example is given by the Mandelbrot set, which can easily be seen to be semicomputable from its very definition, but whose computability is an open problem, related to a conjecture in complex dynamics \cite{Hertling05}.

It turns out that for certain sets, semicomputability is actually equivalent to computability. It was discovered by Miller \cite{Miller02} that finite-dimensional spheres embedded in Euclidean spaces enjoy this property. This result later led Iljazovi\'c to systematically study this computabilty-theoretic property, for spheres embedded in computable metric spaces \cite{Ilja11}, manifolds \cite{Ilja13} and many other sets in a series of articles with several authors \cite{Ilja09,Ilja11,Ilja13,BurnikI14,2018manifolds,CickovicIV19,2020pseudocubes,2020graphs,CelarI21,celar21}. We recently identified which finite simplicial complexes have computable type in \cite{AH22}. Most of the results in the literature are not only about sets, but about pairs. A pair should be informally thought as a set together with its ``boundary'',  the typical example being the~$(n+1)$-dimensional ball and its bounding~$n$-dimensional sphere. A pair has computable type if for any copy~$(X,A)$, if~$X$ and~$A$ are semicomputable then~$X$ is computable.

The purpose of the present article is not to provide new examples, but to develop a structural understanding of the notion of computable type, with several goals in mind. The first goal is practical. Establishing that a space has computable type can be difficult and very technical, so there is a need to provide unifying and simplifying arguments for previous results as well as general tools that can be applied to establish new results with less effort. 

Our general goal is to highlight the interaction between topology and computability underlying the notion of computable type. The definition of this notion combines topology and computability theory, and the arguments always gather ideas from these two fields. It would be clarifying to split the arguments into two parts, a purely topological one and a computability-theoretic one.

Having computable type is by definition a topological invariant, and its precise relationship with other topological invariants should be explored. A more explicit relationship with topology would enable one to leverage the venerable field of topology, notably algebraic topology, in the study of a computability-theoretic problem. Indeed, the main results about computable type rely on classical topological results: Miller used the homology of the complement of the sphere in \cite{Miller02}, Iljazovi{\'c} used some form of Brouwer's fixed-point theorem in \cite{Ilja11,Ilja13}.

Finally, a better theoretical understanding also has many consequences: it can give more information on previous results, give new results with little effort, clarify the role of certain assumptions in the results.

\subsection{Summary of the main results}
Let us present the main results of the article, stated informally.

Mainly two notions of computable type have been introduced in the literature, that consider copies of the set embedded in different classes of topological spaces, namely computable metric spaces and computably Hausdorff spaces respectively. Several results have been established on computable metric spaces \cite{Ilja09,Ilja13} and then extended to computably Hausdorff spaces \cite{CickovicIV19,2018manifolds}. We first show that these two notions are actually equivalent.
\begin{theorem*}[Theorem \ref{thm_hilbert}]
Computable type can be equivalently defined on the Hilbert cube, on computable metric spaces and on computably Hausdorff spaces.
\end{theorem*}
The proof is based on Schr\"oder's effective metrization theorem \cite{1998Schroder} which implies that compact Hausdorff spaces are metrizable in an effective way.

One cannot expect interesting general characterizations of the computable type property: most spaces have computable type simply because they do not have any semicomputable copy. This issue can be fixed by introducing the notion of strong computable type: a compact space~$X$ has strong computable type if for every oracle~$O\subseteq\N$ and every copy~$Y$ of~$X$, if~$Y$ is semicomputable relative to~$O$ then~$Y$ is computable relative to~$O$ (and similarly for pairs). We carry out a thorough analysis of strong computable type. We show that it is equivalent to being minimal for some property of low descriptive complexity (Theorems \ref{thm_minimal} and \ref{thm_sigma2_equiv}). In particular, the next result gives a relatively simple recipe to prove that a space (or pair) has strong computable type.
\begin{theorem*}[Theorem \ref{thm_sigma2}]\label{thm_star_sigma2}
If a pair~$(X,A)$ is minimal satisfying some~$\Sigma^0_2$ topological invariant, then~$(X,A)$ has strong computable type.
\end{theorem*}
The descriptive complexity is measured in the hyperspace of compact sets with the Vietoris or upper Vietoris topology. It is an open question whether Theorem \ref{thm_star_sigma2} is an equivalence, although we do not expect so. As far as we know, being minimal for some~$\Sigma^0_2$ invariant is therefore a distinguished way of having strong computable type, from which particular consequences can be derived, for instance about the level of uniformity of the computation that is involved in strong computable type. If the pair is minimal for some~$\Sigma^0_2$ invariant, then the non-uniformity level of the computation is precisely what is called non-deterministic computability, captured by the Weihrauch degree of choice over~$\N$ (Theorem \ref{thm_choice}).

These structural results have many interesting consequences. We give a purely topological necessary condition, which is sufficient up to some oracle (Corollaries \ref{cor_necessary} and \ref{cor_sct_relative}). We show that for finite simplicial complexes, computable type is equivalent to strong computable type (Corollary \ref{cor_simplicial}). We explain why for some spaces it is difficult to produce a semicomputable copy which is not computable (Theorem \ref{thm_nontrivial}), contrasting with obvious examples such as the line segment or the~$n$-dimensional ball. 

When~$Y$ is a fixed compact ANR and~$X$ is a varying compact space, we show how the set~$[X;Y]$ of continuous functions from~$X$ to~$Y$ quotiented by homotopy can be computed from~$X$, with no computability assumptions about~$Y$ (Theorem \ref{thm_numbering}). A particularly interesting case is when~$Y$ is the~$n$-dimensional sphere~$\SS_n$. This analysis immediately induces topological invariants of low descriptive complexity, which are very classical in topology. 
\begin{theorem*}
The following topological invariants are~$\Sigma^0_2$ in the upper Vietoris topology:
\begin{align*}
(X,A)\in \E_n&\iff  \text{there is a continuous function }f:A\to\SS_n\text{ having no continuous}\\&\qquad\qquad\text{extension }F:X\to \SS_n,\\
X\in\H_n&\iff \text{there is  a continuous function }f:X\to \SS_n\text{ which is not null-homotopic}. 
\end{align*}
\end{theorem*}
We thoroughly study these topological invariants and revisit several results of the literature, showing how they can be obtained by identifying a suitable~$\Sigma^0_2$ invariant for which the space or pair is minimal. These results show that the theory indeed applies, and they strengthen the previous results because they give more information and can be used to derive new results. For instance,
\begin{itemize}
\item The pair~$(\B_{n+1},\SS_n)$ is~$\E_n$-minimal,
\item Every closed~$n$-manifold is~$\H_n$-minimal.
\end{itemize}
We also consider other examples from the literature such as chainable and circularly chainable sets, pseudo-cubes, and as a side result give a few new examples.

Our framework often provides simpler proofs of the results, separating the argument into a computability-theoretic part (showing that a topological invariant is~$\Sigma^0_2$) and a purely topological one (showing that a space is minimal for that invariant), both parts relying heavily on classical topological theorems. 
It is informative to compare the original proof of a previous result with the argument obtained by unfolding the proof using our framework. It turns out that sometimes the underlying arguments are essentially the same -- it is the case for chainable sets and pseudo-cubes. However in other cases, the underlying argument is very different, e.g.~for manifolds. In this case the new argument gives more information, providing for instance a precise measure of the non-uniformity of the computation, and also implying a new result for free, namely that cones of manifolds have (strong) computable type. Note that there is a price to pay: the proof heavily relies on results in algebraic topology about homology and cohomology of manifolds, while the proof given in \cite{Ilja13} is locally and only uses the properties of balls and spheres, notably some form of Brouwer's fixed point theorem.

Let us finally discuss the scope of our approach. Several results in the literature establish a computable type property for non-compact spaces, for instance in \cite{BurnikI14} or \cite{2020graphs}; in this article, we consider compact spaces only. The compactness assumption is central and it is not clear how the theory would extend beyond compact spaces. Some of the previous results used slight variations on the notion of computable type, e.g.~in \cite{Ilja11}, notably defining semicomputability as effective closedness rather than effective compactness, at the price of requiring extra assumptions on the ambient space, such as effective local compactness. In this article, we use the notion of effective compactness which behaves more smoothly.

The article is organized as follows. In Section \ref{sec_background}, we present the standard concepts of computability over topological spaces. In Section \ref{sec_ct} we recall the definition of computable type and prove that some of the variations of this introduced in the literature are actually equivalent. In Section \ref{sec_sct} we introduce the notion of strong computable type and develop the theory. We obtain characterizations using descriptive set theory on the hyperspace of compact subsets of the Hilbert cube. In Section \ref{sec_further} we exploit these characterizations and obtain several results that improve our understanding of strong computable type. In Section \ref{sec_ANR}, we show how the classical notion of Absolute Neighborhood Retract (ANR) behaves well in terms of computability, and allows to define topological invariants of low descriptive complexity, expressing extensibility and null-homotopy of continuous functions to ANRs. In Section \ref{sec_invariants} we study these invariants in more details. We finally apply the whole theory in Section \ref{sec_examples} to revisit many previous results about computable type.

\section{Background on computable topology}\label{sec_background}

We give some background about computable aspects of topological spaces: computable~$T_0$-spaces, computable separation axioms, computability of subsets, the Hilbert cube.

First of all, let us recall enumeration reducibility, which enables one to define a notion of computable reduction between points of countably-based topological spaces. 

\subsection{Enumeration reducibility}
We will mainly use the following notion from computability theory:
a set~$A\subseteq\N$ is \textbf{computably enumerable (c.e.)}
if there exists a Turing machine that, on input~$n\in\N$,
halts if and only if~$n\in A$. An \textbf{enumeration} of~$A$ is any function~$f:\N\to\N$ such that~$A=\{n\in\N:\exists p\in\N, f(p)=n+1\}$. A set is c.e.~if and only if it has a computable enumeration. These notion immediately extends to
subsets of countable sets, whose elements can be encoded by natural
numbers. 

As we will see soon, a point of a countably-based topological space can be identified with a set of natural numbers, namely the set indices of its basic neighborhoods. Therefore, relative computability between points will be conveniently expressed using enumeration reducibility, which we recall now.
\begin{defn}\label{def_enum_red}Let~$A,B\subseteq\N$. We say that~$A$
is \textbf{enumeration reducible} to~$B$, written~$A\leq_{e}B$,
if one of the following equivalent statements holds:
\begin{enumerate}
\item There is an effective procedure producing an enumeration of~$A$
from any enumeration of~$B$,
\item There exists a c.e.~set~$W\subseteq\N$ such that, for all~$x\in\N$,
\begin{equation*}
x\in A\iff\text{there exists a finite set }D\subseteq B\text{ such that }\langle x,D\rangle\in W,
\end{equation*}
\item For every oracle~$O\subseteq\N$, if~$B$ is c.e.~relative to~$O$ then~$A$
is c.e.~relative to~$O$.
\end{enumerate}
\end{defn}
In condition 2.,~$\langle x,D\rangle$ denotes the encoding of the pair of finite objects~$x$ and~$D$ into a natural number. Equivalence between 2.~and 3.~is due to Selman \cite{Selman71}.

\subsection{Computable \texorpdfstring{$T_0$}{T0}-spaces}
\begin{defn}
A \textbf{computable~$T_{0}$-space} is a tuple~$(X,\tau,(B_{i})_{i\in\N})$
where $(X,\tau)$ is a countably-based~$T_{0}$ topological space and~$(B_{i})_{i\in\N}$
is a numbered basis of~$\tau$ such that there exists a c.e.~set~$\mathcal{E}\subseteq\N^{3}$
satisfying~$B_{i}\cap B_{j}=\bigcup_{k:(i,j,k)\in\mathcal{E}}B_{k}.$
\end{defn}

We will often denote a computable~$T_0$-space by~$(X,\tau)$, the numbered basis being implicit. When several topological spaces are involved, we write~$B_{i}^{X}$ for the basis of the space~$X$. A subset~$Y$ of a computable~$T_0$-space~$(X,\tau)$ is also a computable~$T_0$-space, by taking the subspace topology and the numbered basis~$B^Y_i=Y\cap B^X_i$. The product of two computable~$T_0$-spaces is naturally a computable~$T_0$-space.

In a computable~$T_0$-space~$(X,\tau)$, a point can be identified with the set of its basic neighborhoods, therefore the computability properties of a point of such a space is entirely captured by the enumeration degree of its neighborhood basis. This idea is explored in depth by Kihara and Pauly in \cite{KP22}. In particular, a computable reduction between points can be defined using enumeration reducibility as follows.
\begin{defn}\label{def_comp_rel}
Let~$(X,\tau_X)$ and~$(Y,\tau_Y)$ be computable~$T_0$-spaces.
A point~$x$ in~$(X,\tau_X)$ is \textbf{computable relative} to a point~$y$
in~$(Y,\tau_Y)$ if 
\begin{equation*}
\left\{ i\in\N:x\in B_{i}^{X}\right\} \leq_{e}\left\{ i\in\N:y\in B_i^Y\right\} .
\end{equation*}
\end{defn}

A function~$f:X\to Y$ between computable~$T_0$-spaces is \textbf{computable} if and only if for every~$x\in X$,~$f(x)$ is computable relative to~$x$ in a uniform way, i.e.~using the same machine in the reduction. Equivalently,~$f:X\rightarrow Y$ is computable if the sets~$f^{-1}(B_{i}^{Y})$ are effectively open, uniformly in~$i$.

 Let us recall some notions of computability of sets.
\begin{defn}
A set~$A$ in a computable~$T_{0}$-space~$(X,\tau,(B_{i})_{i\in\N})$ is:
\begin{enumerate}
\item \textbf{Effectively compact}, or \textbf{semicomputable}, if it is compact and the set 
\[
\left\{ (i_{1},\ldots,i_{n})\in\N^{\ast}:A\subseteq B_{i_{1}}\cup\ldots\cup B_{i_{n}}\right\} 
\]
 is c.e.,
\item \textbf{Computably overt }if the set~$\left\{ i\in\N:A\cap B_{i}\neq\emptyset\right\} $
is c.e.,
\item \textbf{Computable} if it is effectively compact and computably overt,
\item \textbf{Effectively open} or~$\Sigma^0_1$ if there exists a c.e.~set~$E\subseteq\N$
such that~$A=\bigcup_{i\in E}B_{i}$,
\item \textbf{Effectively closed} or~$\Pi^0_1$ if its complement is effectively open.
\end{enumerate}
A pair~$(A,L)$ is \textbf{semicomputable} if~$A$ and~$L$ are semicomputable.
\end{defn}
We will use the word \emph{semicomputable} when talking about a subset of a space, and \emph{effectively compact} when talking about the space itself.

The image of a (semi)computable set under a computable function is a (semi)computable set.

The next result is simple but very powerful and central in many arguments: closed sets are preserved by taking existential quantification over a compact set, and it holds effectively. Equivalently, open sets are preserved by taking universal quantification over a compact set, effectively so. It is a standard folklore result in computable analysis that can drastically simplify many arguments. It can be found in \cite{Pauly16} for instance, but we include a proof for completeness.
\begin{prop}[Quantifying over a compact space]\label{prop_quant}
Let~$X,Y$ be computable~$T_0$-spaces such that~$Y$ is effectively compact. \begin{itemize}
\item If~$R\subseteq X\times Y$ is effectively closed, then its existential quantification
\begin{equation*}
R^\exists:=\{x\in X:\exists y\in Y,(x,y)\in R\}
\end{equation*}
is effectively closed as well,
\item If~$R\subseteq X\times Y$ is effectively open, then its universal quantification
\begin{equation*}
R^\forall:=\{x\in X:\forall y\in Y,(x,y)\in R\}
\end{equation*}
is effectively open as well.
\end{itemize}
\end{prop}

Of course the two items are equivalent by taking complements, but we make both of them explicit as they are equally useful.
\begin{proof}
We prove the second item, the first one is obtained by taking complements. As~$R$ is effectively open, there exists a c.e.~set~$E\subseteq\N^2$ such that~$R=\bigcup_{(i,j)\in E}B^X_i\times B^Y_j$.  One has
\begin{align}
x\in R^\forall&\iff \{x\}\times Y\subseteq R\\
&\iff \exists \text{ finite set }L\subseteq E, \{x\}\times Y\subseteq \bigcup_{(i,j)\in L}B^X_i\times B^Y_j\label{eq_xY}
\end{align}
because~$\{x\}\times Y$ is compact. Let~$\mathcal{F}$ be the collection of (indices of) finite sets~$L\subseteq E$ such that~$Y\subseteq \bigcup_{(i,j)\in L} B^Y_j$. As~$Y$ is effectively compact,~$\mathcal{F}$ is a c.e.~set. For each~$L\in \mathcal{F}$, let~$U_L=\bigcap_{(i,j)\in L}B^X_i$. The equivalence \eqref{eq_xY} implies that~$R^\forall=\bigcup_{L\in \mathcal{F}}U_L$, which is an effective open set.
\end{proof}

Descriptive set theory and its effective version provide notions of complexity for subsets of topological spaces (\cite{Moscho09}). We have already seen the classes~$\Sigma^0_1$ and~$\Pi^0_1$. In this article, we will also need another class.

\begin{defn}
A set~$A$ in a computable~$T_0$-space~$(X,\tau,(B_i)_{i\in\N})$ is~$\Sigma^0_2$ if it can be expressed as
\begin{equation*}
A=\bigcup_{n\in\N} A_n\setminus B_n
\end{equation*}
where~$A_n,B_n$ are uniformly effective closed sets (i.e.~$\Pi^0_1$).
\end{defn}

Observe that one can equivalently require~$A_n,B_n$ to be effectively \emph{open}, as~$A_n\setminus B_n=(X\setminus B_n)\setminus (X\setminus A_n)$. This class is traditionally defined on computable metric spaces (those spaces are defined in Section \ref{sec_spaces} below), where it is equivalently defined without the sets~$B_n$. The definition given here was proposed for non-Hausdorff spaces  by Scott \cite{Scott76} and Selivanov \cite{Selivanov06}, to make sure that the classes~$\Sigma^0_1$ and~$\Pi^0_1$ are contained in the class~$\Sigma^0_2$.

\subsection{Classes of computable \texorpdfstring{$T_0$}{T0} spaces}\label{sec_spaces}
Each topological separation axiom has a computable version, we recall some of them. These definitions can be found in \cite{1998Schroder} for instance.
\begin{defn}
\label{def: computable topo}
\begin{enumerate}
\item A \textbf{computable metric space} is a tuple~$(X,d,\alpha)$, where~$(X,d)$
is a metric space and~$\alpha:\N\rightarrow X$ is a dense
sequence in~$(X,d)$ such that the real numbers~$d(\alpha_{i},\alpha_{j})$ are uniformly computable,
\item A computable~$T_{0}$-space~$(X,\tau,(B_{i})_{i\in\N})$ is \textbf{computably Hausdorff} if the diagonal~$\Delta=\{(x,x):x\in X\}$ is an effectively
closed subset of~$X\times X$, i.e.~if there exists some c.e.~set~$\D\subseteq\N^{2}$
satisfying
\begin{equation*}
(X\times X)\setminus\Delta=\left\{ (x,y)\in X\times X:x\neq y\right\} =\bigcup_{(i,j)\in\D}B_{i}\times B_{j}.
\end{equation*}
\end{enumerate}
\end{defn}

\begin{rem}
The following facts are standard and can be found in \cite{Pauly16} for instance:
\begin{itemize}
\item A computable metric space is a computable~$T_0$-space, by taking the basis of metric balls~$B_d(\alpha_i,q)$ with~$q>0$ rational; these balls are called the \textbf{rational balls},
\item A computable metric space is computably Hausdorff,
\item In a computably Hausdorff space which is effectively compact, a set is semicomputable (i.e., effectively compact) if and only if it is effectively closed.
\end{itemize}
\end{rem}

The next fact is another powerful feature of effective compactness: in certain situations, if a computable function is injective, then its inverse is automatically computable.
\begin{prop}\label{prop_effective_inverse}Let~$X$ be a computable~$T_{0}$-space,~$Y$
a computably Hausdorff space and~$K\subseteq X$ a semicomputable set. If~$f:K\rightarrow Y$ is a computable injective function, then its inverse~$f^{-1}:f(K)\rightarrow K$ is computable.
\end{prop}
\begin{proof}
It is a folklore result but we include a proof for completeness. We need to show that for a basic open set~$B_i$ of~$X$, its preimage under~$f^{-1}$ is effectively open in~$f(K)$, uniformly in~$i$. The preimage is precisely~$f(B_i\cap K)$. As~$f$ is injective,~$f(B_i\cap K)=f(K)\setminus f(K\setminus B_i)$. As~$K\setminus B_i$ is semicomputable, its image~$f(K\setminus B_i)$ is semicomputable as well. As~$Y$ is computably Hausdorff,~$f(K\setminus B_i)$ is effectively closed, so its complement is effectively open, as wanted.
\end{proof}

\subsubsection{The Hilbert cube}
The Hilbert cube will play a central role in this article, because every computable metric space computably embeds into it, so one can work in this space without loss of generality.
\begin{defn}
The \textbf{Hilbert cube} is the space~$Q=[0,1]^{\N}$ endowed
with the complete metric
\begin{equation*}
d_Q(x,y)=\sum_{i}2^{-i}|x_{i}-y_{i}|.
\end{equation*}
\end{defn}

Here are some important facts about the Hilbert cube.
\begin{fact}
\label{Fact Hilbert} The Hilbert cube~$Q$ satisfies the following properties:
\begin{enumerate}
\item Let~$(\alpha_i)_{i\in\N}$ be a computable enumeration of the points of~$Q$ having rational coordinates, finitely many of them being non-zero. It makes~$(Q,d_Q,\alpha)$ a computable metric space.
\item The Hilbert cube is effectively compact,
\item Therefore, a set~$X\subseteq Q$ is semicomputable if and only if its complement is an effective open set,
\item Every computable metric space embeds effectively into the Hilbert
cube. More precisely, for every computable metric space~$(X,d,\alpha)$ there
exists a computable embedding~$f:X\to Q$ such that~$f^{-1}$
is computable, defined as~$f(x)=(1/(1+d(x,\alpha_i)))_{i\in\N}$.
\end{enumerate}
\end{fact}

In particular, the fourth item implies that the semicomputable subsets of arbitrary computable metric spaces are no more general than the semicomputable subsets of~$Q$.

\section{Computable type and the Hilbert cube}\label{sec_ct}

The notion of computable type takes its origins in an article by Miller \cite{Miller02}, was studied by Iljazovi{\'c} et al. in \cite{Ilja09,Ilja11,Ilja13,BurnikI14,2018manifolds,CickovicIV19,2020pseudocubes,2020graphs,CelarI21,celar21} and by the authors in \cite{AH22}. The first formal definition as given in \cite{2018manifolds}.

Let us present the language needed to formulate definition of computable type.
\begin{defn}
\label{def:A-pair}A \textbf{pair}~$(X,A)$ consists of a compact
metrizable space~$X$ and a compact subset~$A\subseteq X$. A \textbf{copy}
of a pair~$(X,A)$ in a topological space~$Z$ is a pair~$(Y,B)$
such that~$Y\subseteq Z$ is homeomorphic to~$X$ and~$A$ is sent
to~$B$ by the homeomorphism.
\end{defn}

If~$(X,A)$ and~$(Y,B)$ are two pairs embedded in a topological space~$Z$, then we say that~$(X,A)$ is a \textbf{subpair} of~$(Y,B)$, or is \textbf{contained} in~$(Y,B)$ if~$X\subseteq Y$ and~$A\subseteq B$. We then write~$(X,A)\subseteq (Y,B)$. We say that~$(X,A)$ is \textbf{strictly contained} in~$(Y,B)$, or is a \textbf{proper subpair} of~$(Y,B)$, if in addition~$X\neq Y$ (note that~$A$ may equal~$B$), and we write~$(X,A)\subsetneq (Y,B)$.

\begin{defn}
If~$(X,A)$ and~$(Y,B)$ are two pairs, then a \textbf{function between pairs}~$f:(X,A)\to (Y,B)$ is a function~$f:X\to Y$ such that~$f(A)\subseteq B$.
\end{defn}

\begin{defn}\label{def_ct}
A compact metrizable space~$X$ has \textbf{computable type} if for every copy~$Y$ in the Hilbert cube, if~$Y$ is semicomputable then~$Y$ is computable.

A compact pair~$(X,A)$ has \textbf{computable type} if for every copy~$(Y,B)$ of~$(X,A)$ in the Hilbert cube, if~$(Y,B)$ is semicomputable then~$Y$ is computable.
\end{defn}

Note that the notion for pairs subsumes the notion for single sets, by considering the pair~$(X,\emptyset)$.

Originally, two notions of computable type were studied in \cite{Ilja09} and \cite{CickovicIV19} respectively, using other spaces than the Hilbert cube. Precisely,
\begin{itemize}
\item A pair~$(X,A)$ has \textbf{computable type on computable metric spaces} if for every copy~$(Y,B)$ of the pair in any computable metric space~$Z$, if~$(Y,B)$ is semicomputable then~$Y$ is computable,
\item A pair~$(X,A)$ has \textbf{computable type on computably Hausdorff spaces} if for every copy~$(Y,B)$ of the pair in any computably Hausdorff space~$Z$, if~$(Y,B)$ is semicomputable then~$Y$ is computable.
\end{itemize}

In this section, we prove that the distinction between computable type on computable metric spaces and computably Hausdorff spaces is unnecessary, because they are actually equivalent. Moreover, it is sufficient to consider the Hilbert cube only, as in Definition \ref{def_ct}. Therefore there is no more need to distinguish between these definitions, and the results about computable type on computable metric spaces immediately extend to computably Hausdorff spaces (for instance, the results in \cite{Ilja09,Ilja13} imply the results in \cite{CickovicIV19,2018manifolds}).

\begin{thm}\label{thm_hilbert}
For a pair~$(X,A)$, the following statements are equivalent:
\begin{enumerate}
\item $(X,A)$ has computable type (on the Hilbert cube),
\item $(X,A)$ has computable type on computable metric spaces,
\item $(X,A)$ has computable type on computably Hausdorff spaces.
\end{enumerate}
\end{thm}

We now proceed with the proof of Theorem \ref{thm_hilbert}. The implications~$3.\Rightarrow 2.\Rightarrow 1.$ are straightforward, we only need to prove~$1.\Rightarrow 3$.  The idea is that a compact subspace of a Hausdorff space is metrizable, so embedding a compact set in a Hausdorff space implicitly induces an embedding in a metric space. We need to make this argument effective. We will use Schr\"oder's effective version of Urysohn metrization theorem \cite{1998Schroder}, so we need to introduce the computable version of regular spaces.
\begin{defn}
A computable~$T_0$-space~$(X,\tau)$ is \textbf{computably regular} if there is a computable procedure associating to each basic open set~$B_{n}$ a sequence of effective open sets~$(U_k)_{k\in\N}$ and a sequence of effective closed sets~$(F_k)_{k\in\N}$ such that~$U_{k}\subseteq F_{k}\subseteq B_{n}$ and~$B_{n}=\bigcup_{k}U_{k}$.
\end{defn}

\begin{prop}
\label{prop copy regular}Let~$(X,\tau)$ be a computable~$T_0$-space. If~$(X,\tau)$ is effectively compact and computably Hausdorff, then it is computably regular.
\end{prop}

\begin{proof}
Let~$\D\subseteq\N^2$ witness that the space is computably Hausdorff. Let~$n\in\N$. The set~$K=X\setminus B_{n}$ is effectively
compact, so one can compute an enumeration~$(L_{k})_{k\in\N}$
of all the finite sets~$L\subseteq \D$ such that~$K\subseteq\bigcup_{(i,j)\in L}B_{j}$.
To each such~$L_{k}$ associate the pair~$(U_{k},F_{k})$ defined
by
\begin{align*}
U_{k} & =\bigcap_{(i,j)\in L_{k}}B_{i}\\
F_{k} & =X\setminus\bigcup_{(i,j)\in L_{k}}B_{j}.
\end{align*}
By construction,~$U_{k}$ and~$F_{k}$ are uniformly effectively
open and closed respectively.

Next,~$U_{k}$ is contained in~$F_{k}$ because~$B_{i}$ is disjoint
from~$B_{j}$ for all~$(i,j)\in L_{k}\subseteq \D$. As~$X\setminus B_{n}\subseteq\bigcup_{(i,j)\in L_{k}}B_{j}$,
one has~$F_{k}=X\setminus\bigcup_{(i,j)\in L_{k}}B_{j}$$\subseteq B_{n}$.

Finally, we show that~$B_{n}=\bigcup_{k}U_{k}$. Let~$x\in B_{n}$.
As~$x\notin K$, the compact set~$\{x\}\times K$ is contained in~$\bigcup_{(i,j)\in D}B_{i}\times B_{j}$,
so there exists a finite set~$L\subseteq \D$ such that~$\{x\}\times K\subseteq\bigcup_{(i,j)\in L}B_{i}\times B_{j}$.
This inclusion is still satisfied if one removes from~$L$ the pairs~$(i,j)$
such that~$x\notin B_{i}$. Let~$L^{\prime}\subseteq L$ be the
result of this operation. There exists~$k$ such that~$L^{\prime}=L_{k}$,
and~$x\in U_{k}$.
\end{proof}

Schr\"oder proved the following effective Urysohn metrization theorem (Theorem 6.1 in \cite{1998Schroder}) : every computably regular space~$X$ admits a computable metric, i.e.~a computable function~$d:X\times X\to \R$ which is a metric that induces the topology of the space (note that it does not mean that~$X$ is a computable metric space, because it may not contain a dense computable sequence). The proof of that result implies the next lemma.
\begin{lem}
\label{lem:effective lemma}If~$(X,\tau)$ is a computably regular space, then there exists a computable injection~$h:X\rightarrow Q$.
\end{lem}
\begin{proof}
The proof of Theorem 6.1 in \cite{1998Schroder} consists in building a computable sequence of functions~$g_i:X\to [0,1]$ that separates points, i.e.~such that for all~$x,y\in X$ with~$x\neq y$,~$g_i(x)\neq g_i(y)$ for some~$i\in\N$. The metric~$d$ is then defined as~$d(x,y)=\sum_i 2^{-i}|g_i(x)-g_i(y)|$. In our case, we simply define~$h(x)=(g_i(x))_{i\in\N}$.
\end{proof}

\begin{cor}
\label{cor embedding}If~$K$ is a semicomputable set in a computably Hausdorff space~$(X,\tau)$, then there exists a computable embedding~$h:K\to Q$ whose inverse is computable.
\end{cor}

\begin{proof}
Consider the subspace~$(K,\tau_K)$ with the induced topology~$\tau_K=\{U\cap K:U\in\tau\}$. Because it is a subspace of a computably Hausdorff space, it is computably Hausdorff as well. As~$K$ is semicomputable, as a set, it is effectively compact as a space. Therefore, Proposition \ref{prop copy regular} implies that~$(K,\tau_K)$ is computably regular. By Lemma \ref{lem:effective lemma}, there exists a computable injection~$h:K\to Q$. By Proposition \ref{prop_effective_inverse}, its inverse is computable.
\end{proof}
Now, we prove our theorem.
\begin{proof}[Proof of Theorem \ref{thm_hilbert}]
We prove~$1.\Rightarrow 3$. If~$(K,L)$ is a pair that has computable type (on the Hilbert cube),~$(X,\tau)$ is a computably Hausdorff space and~$f:K\to X$ is an
embedding such that~$f(K)$ and~$f(L)$ are semicomputable,
then by Corollary \ref{cor embedding}, there exists a computable embedding~$h:f(K)\to Q$
such that~$h^{-1}$ is computable. Hence,~$h\circ f:K\to Q$
is an embedding such that~$h\circ f(K)$ and~$h\circ f(L)$ are
semicomputable. As~$(K,L)$ has computable type, the set~$h\circ f(K)$ is computable. The set~$f(K)$ is the image of the computable set~$h\circ f(K)$ by the computable function~$h^{-1}$, so~$f(K)$ is a computable set. As a result,~$(K,L)$ has computable type on computably Hausdorff spaces.
\end{proof}
%

\section{Strong computable type}\label{sec_sct}

The notion of computable type suffers from a severe drawback: if a compact metrizable space has no semicomputable copy in the Hilbert cube, then it vacuously has computable type. As there are only countably many semicomputable sets, most compact metrizable spaces have computable type for no good reason. In particular, there is no hope to obtain interesting characterizations of this notion.

This observation leads us to define a stronger and more robust notion: having computable type relative to any oracle. This definition solves the previous issue and has the advantage of  lending itself to topological analysis, notably in terms of the topologies on the hyperspace~$\K(Q)$ of compact subsets of~$Q$, which enables us to obtain topological characterizations.

\subsection{Definition}
We first define the notion of strong computable type, which is a simple relativization of Definition \ref{def_ct}.
\begin{defn}
A compact metrizable space~$X$ has \textbf{strong computable type} if for every oracle~$O$ and every copy~$Y$ of~$X$ in the Hilbert cube, if~$Y$ is semicomputable relative to~$O$, then~$Y$ is computable relative to~$O$.

A pair~$(X,A)$ has \textbf{strong computable type} if for every oracle~$O$ and every copy~$(Y,B)$ of~$(X,A)$ in~$Q$, if~$(Y,B)$ is semicomputable relative to~$O$, then~$Y$ is computable relative to~$O$.
\end{defn}
%

\begin{rem}
All the spaces which were proved to have computable type in the literature \cite{Miller02, Ilja09, Ilja11, Ilja13,BurnikI14,2018manifolds,CickovicIV19,2020pseudocubes,2020graphs,CelarI21,celar21,AH22} actually have strong computable type because the proofs hold relative to any oracle.

Note that the proof that the definition of computable type reduces to the Hilbert cube (Theorem \ref{thm_hilbert}) also extends to strong computable type, so defining strong computable type on computable metric spaces or computably Hausdorff spaces would yield equivalent notions.
\end{rem}

We will see in Section \ref{sec_surjection} that for finite simplicial complexes, strong computable type is equivalent to computable type (Corollary \ref{cor_simplicial}), refining the results in \cite{AH22}. More generally, we expect that for natural spaces, the notion of strong computable type is actually no stronger than the notion of computable type.


Intuitively, a pair has strong computable type iff for every copy~$(Y,B)$ in the Hilbert cube, the set~$Y$ can be fully computed if we are only given the compact information about~$Y$ and~$B$. We make it precise by seeing~$Y$ and~$B$ as points of the hyperspace of compact subsets of the Hilbert cube with suitable topologies, as follows.
\begin{defn}
Let~$\K(Q)$ be the \textbf{hyperspace} of~$Q$, i.e.~the space
of compact subsets of~$Q$. It can be equipped with:
\begin{enumerate}
\item The \textbf{upper Vietoris} topology\textbf{~$\upV$} generated by the sets of the form~$\{K\in\K(Q):K\subseteq U\}$, where~$U$ ranges over the open subsets of~$Q$,
\item The \textbf{lower Vietoris} topology\textbf{~$\lowV$ }generated by the sets of the form~$\{K\in\K(Q):K\cap U\neq\emptyset\}$, where~$U$ ranges over the open subsets of~$Q$,
\item The \textbf{Vietoris} topology~$\V$ generated by~$\upV$ and~$\lowV$.
\end{enumerate}
\end{defn}

These three topological spaces are computable~$T_0$-spaces. A countable basis of~$\upV$ is obtained by taking~$U$ among the finite unions of rational balls of~$Q$. A countable subbasis of~$\lowV$ is obtained by taking~$U$ among the rational balls of~$Q$, and adding~$\K(Q)$ to the subbasis. The space~$(\K(Q),\V)$ is also a computable metric space, by taking the Hausdorff metric, and the dense sequence of finite sets of rational points of~$Q$ with finitely many non-zero coordinates. Moreover, these three spaces are effectively compact (see Section~$5$ in~\cite{IK20} for instance; it is sufficient to prove it for the stronger topology~$\V$).

Each computability notion of compact set can then be seen as the notion of computable point in one of these topologies. In particular, for a compact set~$K\subseteq Q$,
\begin{align*}
K \text{ is semicomputable }\iff K\text{ is a computable element of }(\K(Q),\upV),\\
K\text{ is computable }\iff K\text{ is a computable element of }(\K(Q),\lowV).
\end{align*}
The topology~$\lowV$ induces the notion of computably overt set, which will not be discussed in this article. It is easy to get confused with the many different uses of the word ``computable'', but we will always make it clear what meaning is used. 

The property of having strong computable type can then be rephrased as relative computability between elements of~$\K(Q)$ with various topologies, using the notion of relative computability given by Definition \ref{def_comp_rel}. As relative computability is expressed in terms of enumeration reducibility, Definition \ref{def_enum_red} then provides equivalent formulations of this notion, which we will implicitly use in the rest of the article. Let us state such a formulation for clarity.

\begin{prop}
For a pair~$(X,A)$, the following statements are equivalent:
\begin{enumerate}
\item $(X,A)$ has strong computable type,
\item For every copy~$(Y,B)$ of~$(X,A)$ in~$Q$, the element~$Y$ of the space~$(\K(Q),\V)$ is computable relative to the element~$(Y,B)$ of the product space~$(\K^{2}(Q),\upVV)$.
\end{enumerate}
\end{prop}
\begin{proof}
The second statement is about relative computability, which is defined as enumeration reducibility between neighborhood bases of~$Y$ and~$(Y,B)$. The definition of strong computable type is precisely formulation (3) of enumeration reducibility (Definition \ref{def_enum_red}).
\end{proof}

It will have interesting consequences, because it enables one to use topology and descriptive set theory on~$\K(Q)$ to analyze a computability-theoretic property.

\subsection{A first characterization}\label{sec_first_char}
We give a first characterization of strong computable type. This result is related to the characterization obtained by Jeandel in \cite{Jeandel17} of the sets~$A\subseteq\N$ that are \emph{total}, i.e.~whose complement~$\N\setminus A$ is enumeration reducible to~$A$.

We first need two definitions.
\begin{defn}
A \textbf{property of pairs} is a subset~$\P$ of~$\K(Q)\times\K(Q)$. It is a \textbf{topological invariant}, or shortly an \textbf{invariant}, if for every~$(X,A)\in \P$,
every copy~$(Y,B)$ of~$(X,A)$ is in~$\P$.
\end{defn}
We recall that~$(Y,B)$ is a proper subpair of~$(X,A)$, written~$(Y,B)\subsetneq (X,A)$, if~$Y\subsetneq X$ and~$B\subseteq A$.

\begin{defn}
Let~$\P$ be a property of pairs. A pair~$(X,A)\subseteq Q$ is \textbf{$\P$-minimal} if~$(X,A)\in \P$ and for every proper subpair~$(Y,B)\subsetneq (X,A)$, one has~$(Y,B)\notin \P$.
\end{defn}

We can now state our first characterization of strong computable type.
\begin{thm}
\label{thm_minimal}For a pair~$(X,A)$, the following statements are equivalent:
\begin{enumerate}
\item $(X,A)$ has strong computable type,
\item For every copy~$(Y,B)$ of~$(X,A)$ in~$Q$, there exists a property~$\P$ of pairs which is~$\Pi^0_1$ in the topology~$\upVV$ and such that~$(Y,B)$ is~$\P$-minimal.
\end{enumerate}
\end{thm}

\begin{proof}
We fix an arbitrary copy of~$(X,A)$ in~$Q$ (that we still call~$(X,A)$) and prove that the following equivalence holds for that copy: the set of basic~$\V$ neighborhoods of~$X$ is enumeration reducible to the set of basic~$\upVV$ neighborhoods of~$(X,A)$ iff there exists a~$\Pi^0_1$ property~$\P$ such that~$(X,A)$ is~$\P$-minimal.

$2.\Rightarrow 1$. Assume that~$\P$ is~$\Pi^0_1$ in~$\upVV$ and that~$(X,A)$ is~$\P$-minimal. Given the~$\upVV$ neighborhoods of~$(X,A)$, we need to enumerate the rational balls~$U$ intersecting~$X$. Note that~$U$ intersects~$X$ iff~$X\setminus U$ is a proper subset of~$X$; as~$(X,A)$ is~$\P$-minimal, it is equivalent to~$(X\setminus U,A\setminus U)\notin \P$. Given~$(X,A)$ in the topology~$\upVV$, one can compute~$(X\setminus U,A\setminus U)$ in the topology~$\upVV$ (indeed,~$X\setminus U\subseteq V\iff X\subseteq U\cup V$, and similarly for~$A\setminus U$) so one can semi-decide whether~$(X\setminus U,A\setminus U)\notin\P$, i.e.~whether~$U$ intersects~$X$.

$1.\Rightarrow 2$. Consider a machine~$M$ which takes any enumeration of the basic~$\upVV$-neighborhoods of~$(X,A)$ and enumerates the open balls intersecting~$X$. Let~$\U$ be the set of pairs~$(X',A')$ on which~$M$ fails in the following sense: after reading a finite sequence of~$\upVV$ neighborhoods of the pair~$(X',A')$, the machine enumerates an open ball~$U$ such that~$X'\cap\overline{U}=\emptyset$, where~$\overline{U}$ is the corresponding closed ball. Let~$\P$ be the complement of~$\U$. Note that~$(X,A)\in\P$ because the machine does not fail on~$(X,A)$.

We first show that~$\U$ is effectively open in~$\upVV$. If~$\sigma=\sigma_0\ldots\sigma_k$ is a finite sequence of basic~$\upVV$-open sets and~$U$ is a basic open subset of~$Q$, then let
\begin{equation*}
\U_{(\sigma,U)}=\{(X',A'):\forall i\leq k, (X',A')\in\sigma_i\text{ and }X'\subseteq Q\setminus\overline{U}\}.
\end{equation*}
It is an effective~$\upVV$-open set, and~$\U$ is the union of all the~$\U_{(\sigma,U)}$ such that~$M$ outputs~$U$ after reading~$\sigma$, so it is a c.e.~union. Therefore~$\U$ is an effective~$\upVV$-open set and~$\P$ is~$\Pi^0_1$ in~$\upVV$.

We now show that~$(X,A)$ is~$\P$-minimal. If~$(X',A')$ is a proper subpair of~$(X,A)$, then there exists a ball~$U$ intersecting~$X$ such that~$\overline{U}$ is disjoint from~$X'$. On an arbitrary enumeration of the~$\upVV$-neighborhoods of~$(X,A)$, the machine eventually outputs~$U$ after reading a finite sequence~$\sigma$. As~$(X',A')\subseteq (X,A)$, the~$\upVV$-neighborhoods of~$(X,A)$ are also~$\upVV$-neighborhoods of~$(X',A')$ so the machine fails on~$(X',A')$. Therefore,~$(X',A')\in \U$, i.e.~$(X',A')\notin\P$. We have shown that~$(X,A)$ is~$\P$-minimal.
\end{proof}

In particular, this characterization immediately implies a simple sufficient condition for having strong computable type.
\begin{thm}\label{thm_sigma2}Let~$\P$ be a topological invariant which is~$\Sigma_{2}^{0}$
in~$\upVV$. Every minimal element of~$\P$ has strong computable type.
\end{thm}

\begin{proof}
Let~$\P=\bigcup_{n}\P_n\setminus \Q_n$ where each~$\P_n,\Q_n$ is~$\Pi_{1}^{0}$ in~$\upVV$. Assume that~$(X,A)$ is~$\P$-minimal and let~$(Y,B)$ be a copy of~$(X,A)$ in~$Q$. This copy belongs to some~$\P_n\setminus \Q_n$. Let us show that~$(Y,B)$ must be~$\P_n$-minimal. Let~$(Y',B')$ be a proper subpair of~$(Y,B)$. The set~$\Q_n$ is~$\Pi^0_1$ in~$\upVV$ so it is an upper set. As~$(Y,B)\notin\Q_n$,~$(Y',B')\notin \Q_n$ as well. As~$(Y,B)$ is~$\P$-minimal,~$(Y',B')\notin \P$, so~$(Y',B')\notin \P_n$. Therefore,~$(Y,B)$ is~$\P_n$-minimal.

We have shown that each copy~$(Y,B)$ of~$(X,A)$ is~$\P_n$-minimal for some~$n$, so we can apply Theorem \ref{thm_minimal}, implying that~$(X,A)$ has strong computable type.
\end{proof}

We will see applications of this result in the sequel. Theorem \ref{thm_sigma2} is only an implication and we will see later that the converse implication does not hold in general. However, in the next section we obtain a characterization of strong computable type using~$\Sigma^0_2$ invariants and a weak form of minimality.

\paragraph{Discussion about minimality}

When applying Theorem \ref{thm_minimal} and Theorem \ref{thm_sigma2} to prove that a pair has strong computable type, one needs to show that a pair~$(X,A)$ is~$\P$-minimal for some~$\P$. The definition of minimality involves all the proper subpairs~$(Y,B)\subsetneq (X,A)$. However, because~$\P$ is~$\Pi^0_1$ in~$\upVV$, it is not necessary to consider all of them.

The reason is that the lower and upper Vietoris topologies interact nicely with the inclusion ordering on compact sets:
\begin{itemize}
\item If a set~$\P\subseteq\K(Q)$ is closed in~$\upV$ or open in~$\lowV$, then it is an \textbf{upper set}: if~$K\subseteq K'\subseteq Q$ and~$K\in\P$, then~$K'\in\P$.
\item Dually, if~$\P$ is open in the topology~$\upV$ or closed in the topology~$\lowV$, then it is a \textbf{lower set}: if~$K\subseteq K'\subseteq Q$ and~$K'\in\P$, then~$K\in\P$.
\end{itemize}

It gives a shortcut to prove that a pair~$(X,A)$ is~$\P$-minimal, when~$\P$ is an upper set, for instance when~$\P$ is~$\Pi^0_1$ or more generally closed in the topology~$\upVV$.

\begin{defn}\label{def_cofinal}
Let~$(X,A)$ be a pair and~$\F$ a family of proper subpairs of~$(X,A)$. We say that~$\F$ is \textbf{cofinal} if every proper subpair of~$(X,A)$ is contained in some pair in~$\F$.
\end{defn}

The next result is elementary but very useful.
\begin{lem}\label{lem_cofinal}
Let~$\P$ be a property of pairs which is an upper set. Let~$(X,A)\in\P$ and~$\F$ be cofinal. 
If no pair of~$\F$ is in~$\P$ then~$(X,A)$ is~$\P$-minimal.
\end{lem}
\begin{proof}
If~$(Y,B)\subsetneq (X,A)$, then~$(Y,B)\subseteq (Z,C)\subsetneq (X,A)$ for some~$(Z,C)\in\F$, by cofinality. As~$(Z,C)\notin \P$ by assumption and~$\P$ is an upper set,~$(Y,B)\notin\P$.
\end{proof}

We will implicitly apply this observation, notably with the following families:
\begin{itemize}
\item The family~$\{(Y,Y\cap A):Y\subsetneq X\}$ is cofinal,
\item If~$\mathcal{B}$ is a basis of the topology of~$X$ consisting of non-empty open sets, then the family~$\{(X\setminus B,A\setminus B):B\in\mathcal{B}\}$ is cofinal,
\item If~$A$ has empty interior in~$X$, then the family~$\{(Y,A):A\subseteq Y\subsetneq X\}$ is cofinal.
\end{itemize}

%
%
%
%
%

\subsection{A characterization using invariants}
In order to prove that a pair has strong computable type using Theorem \ref{thm_minimal}, one needs to find a~$\Pi^0_1$ property for each copy. In this section we improve the result by making the property less dependent on the copy. It is possible because there are countably many~$\Pi^0_1$ properties, but uncountably many copies of a pair, so some property must work for ``many'' copies. The word ``many'' can be made precise by using Baire category on the space of continuous functions of the Hilbert cube to itself.


\begin{thm}\label{thm_sigma2_equiv}For a pair~$(X,A)$, the following statements are equivalent:
\begin{enumerate}
\item $(X,A)$ has strong computable type,
\item There exists a~$\Sigma_{2}^{0}$ invariant~$\bigcup_{n\in\N}\P_n$, where the sets~$\P_n$ are uniformly~$\Pi^0_1$ in~$\upVV$, such that every copy of~$(X,A)$ in~$Q$ is~$\P_n$-minimal for some~$n$.
\end{enumerate}
\end{thm}

Note that the invariant is more than a $\Sigma^0_2$ property, because it is a union of closed sets rather than a union of \emph{differences} of closed sets. The rest of the section is devoted to the proof of this result.

The idea of the proof is that one of the~$\Pi^0_1$ properties from Theorem \ref{thm_minimal} works for ``many'' copies of~$(X,A)$, in the sense of Baire category. We then transform this~$\Pi^0_1$ property into a~$\Sigma^0_2$ invariant. This transformation is an instance of Vaught's transform \cite{Vaught74}. Let us briefly explain this transform. There are several ways of defining an invariant out of a property~$\P$. For instance, the set of pairs having at least one copy in~$\P$ is an invariant; the set of pairs whose copies are all in~$\P$ is another invariant. However, these two invariants can have very high descriptive complexities (typically~$\Sigma^1_1$ and~$\Pi^1_1$ respectively). Vaught's transform is an intermediate way of converting a property into an invariant which almost preserves the descriptive complexity: the invariant is the set of pairs  having ``many'' copies in~$\P$, where ``many'' is expressed using Baire category. We now give the details.

We endow the space~$\con(Q)$ of continuous functions from~$Q$ to itself with the complete separable metric
\begin{equation*}
\rho(f,g)=\max_{x\in Q}d_Q(f(x),g(x)).
\end{equation*}

If functions~$f,g$ are defined on~$X\subseteq Q$ only, then we also define~$\rho_X(f,g)=\max_{x\in X}d_Q(f(x),g(x))$.
\begin{rem}We will often use the following inequalities:
\begin{align}
\rho(f\circ h,g\circ h)&\leq \rho(f,g),\label{ineq1_rho}\\
\rho(f\circ g,\id_Q)&\leq \rho(f,\id_Q)+\rho(g,\id_Q).\label{ineq_rho}
\end{align}
The second inequality holds because~$\rho(f\circ g,\id_Q)\leq \rho(f\circ g,g)+\rho(g,\id_Q)\leq \rho(f,\id_Q)+\rho(g,\id_Q)$.
\end{rem}

We can make~$(\con(Q),\rho)$ a computable metric space.
\begin{lem}\label{lem_dense}
There exists a computable sequence of injective continuous functions~$\phi_i:Q\to Q$ that is dense in the metric~$\rho$.
\end{lem}
\begin{proof}
We first build a dense computable sequence of functions~$f_i:Q\to Q$ that are not necessarily injective. Say that an element~$x\in Q$ is \emph{dyadic of order}~$n$ if it has the form~$x=(x_0,\ldots,x_{n-1},0,0,0,\ldots)$ where each~$x_i$ is a multiple of~$2^{-n}$. For each~$n\in\N$, the finite set of dyadic elements of order~$n$ forms a regular grid. One can then define a piecewise affine map by assigning a dyadic element to each dyadic element of order~$n$ and interpolating affinely in between. All the possible such assignments provide a dense computable sequence of functions from~$Q$ to itself.

Every continuous function~$f:Q\to Q$ can be approximated by injective continuous functions~$g_n:Q\to Q$ defined as follows: let~$g_n(x)$ be the sequence starting with the first~$n$ terms of~$f(x)$, appended with~$x$. If~$f$ is computable then the functions~$g_n$ are computable, uniformly. Therefore injective approximations of the functions~$f_i$ give a computable dense sequence of injective functions.
\end{proof}

\begin{defn}
Let~$\epsilon>0$. An~\textbf{$\epsilon$-function} is a continuous function~$f:Q\to Q$ such that~$\rho(f,\id_Q)<\epsilon$. An~\textbf{$\epsilon$-deformation} of a pair~$(X,A)\subseteq Q$ is the image~$f(X,A):=(f(X),f(A))$ of~$(X,A)$ under an~$\epsilon$-function~$f$.
\end{defn}
Note that if~$f$ is an~$\epsilon$-function and~$g$ is a~$\delta$-function, then~$f\circ g$ is an~$(\epsilon+\delta)$-function. Therefore, an~$\epsilon$-deformation of a~$\delta$-deformation of~$(X,A)$ is a~$(\epsilon+\delta)$-deformation of~$(X,A)$.

\begin{rem}
The notion of~$\epsilon$-function can be defined for partial functions~$f:X\to Q$, where~$X\subseteq Q$, by requiring~$\rho_X(f,\id_X)<\epsilon$. The notion of~$\epsilon$-deformation of~$(X,A)$ does not change if we consider~$\epsilon$-functions define on~$Q$ or on~$X$ only, so we will freely use this flexibility in the sequel.

Indeed, if~$f:X\to Q$ is an~$\epsilon$-function, then~$f$ has a continuous extension~$\tilde{f}:Q\to Q$ which is an~$\epsilon$-function. The function~$\tilde{f}$ can be defined as follows: let~$f_0:Q\to Q$ be a continuous extension of~$f$ obtained by the Tietze extension theorem; let~$\delta>0$ and~$\tilde{f}(x)=(1-\lambda_x)f_0(x)+\lambda_x x$, where~$\lambda_x=\min(1,d_Q(x,X)/\delta)$ (note that~$Q$ indeed allows convex combinations). $\tilde{f}$ is a continuous extension of~$f$ and if~$\delta$ is sufficiently small, then~$\rho(\tilde{f},\id_Q)<\epsilon$.
\end{rem}

If a homeomorphism~$f:X\to Y$ is an~$\epsilon$-function, then its inverse~$f^{-1}:Y\to X$ is an~$\epsilon$-function as well.

\begin{defn}\label{def_Pstar}
Let~$\P$ be a property of pairs. We define an invariant~$\P^*$ as follows:~$(X,A)\in \P^*$ iff there exists a copy~$(X',A')$ of~$(X,A)$ and an~$\epsilon>0$ such that every~$\epsilon$-deformation of~$(X',A')$ is in~$\P$. 
\end{defn}

The key observation about~$\P^*$ is that its descriptive complexity is not too high, compared to the descriptive complexity of~$\P$: if~$\P$ is~$\Pi^0_1$, then~$\P^*$ is~$\Sigma^0_2$.

\begin{rem}
When~$\P$ is~$\Pi^0_1$ (more generally closed) in~$\VV$,~$\P^*$ could be equivalently defined by considering \emph{injective}~$\epsilon$-functions only, in which case~$\epsilon$-deformations of~$(X',A')$ would all be copies of~$(X',A')$. Indeed, the injective continuous functions are dense in~$\con(Q)$, so every~$\epsilon$-deformation of~$(X',A')$ is a limit, in the Vietoris topology, of~$\epsilon$-deformations by injective~$\epsilon$-functions; if the latter are all in~$\P$, then every~$\epsilon$-deformation is also in~$\P$ as~$\P$ is closed in~$\VV$.
\end{rem}

\begin{proof}[Proof of Theorem \ref{thm_sigma2_equiv}]
For a rational~$\epsilon>0$, we define
\begin{equation*}
\P_\epsilon=\{(X,A)\subseteq Q:\text{every~$\epsilon$-deformation of~$(X,A)$ is in~$\P$}\}.
\end{equation*}
One has~$\P^*=\{(X,A)\subseteq Q:\exists \epsilon>0, (X,A)\text{ has a copy in }\P_\epsilon\}$.

\begin{claim}
Let~$\tau$ be either~$\VV$ or~$\upVV$. If~$\P$ is~$\Pi^0_1$ in~$\tau$ then $\P_\epsilon$ is~$\Pi^0_1$ in~$\tau$.
\end{claim}
\begin{proof}
The function~$f\in\con(Q)\mapsto f(X,A)\in(\K^2(Q),\tau)$ is continuous,~$\P$ is closed in~$\tau$ and the functions~$\phi_k$ are dense in~$\con(Q)$. Therefore, the quantification over the~$\epsilon$-functions can be replaced by a quantification over the~$k$'s such that~$\phi_k$ is an~$\epsilon$-function:~$(X,A)\in\P_\epsilon$ iff~$\forall k\in\N$,~$\rho(\phi_k,\id_Q)<\epsilon$ implies~$\phi_k(X,A)\in\P$. It is a~$\Pi^0_1$ predicate.
\end{proof}

Say that two pairs~$(X,A)$ and~$(Y,B)$ in~$Q$ are~$\epsilon$-homeomorphic if there exists a homeomorphism~$f:(X,A)\to (Y,B)$ which is an~$\epsilon$-function (note that this relation is symmetric, because in this case~$f^{-1}$ is an~$\epsilon$-function as well).
\begin{claim}
If~$(X_0,A_0)\subseteq Q$ and~$(X_1,A_1)\subseteq Q$ are homeomorphic, then for every~$\epsilon>0$ there exists~$i$ such that~$\phi_i(X_0,A_0)$ and~$(X_1,A_1)$ are~$\epsilon$-homeomorphic.
\end{claim}
\begin{proof}
Let~$f:(X_0,A_0)\to (X_1,A_1)$ be a homeomorphism and let~$i$ be such that~$\rho_{X_0}(\phi_i,f)<\epsilon$. Let~$\psi=\phi_i\circ f^{-1}$. One has~$\phi_i(X_0,A_0)=\psi(X_1,A_1)$ and~$\rho_{X_1}(\psi,\id_Q)=\rho_{X_0}(\phi_i,f)<\epsilon$.
\end{proof}

\begin{claim}
Let~$\tau$ be either~$\VV$ or~$\upVV$. If~$\P$ is~$\Pi^0_1$ in~$\tau$ then $\P^*$ is~$\Sigma^0_2$ in~$\tau$.
\end{claim}
\begin{proof}
Again, the idea is that we can replace the quantification over the copies, i.e.~over the injective continuous functions, by a quantification over the~$\phi_i$'s. Precisely, let
\begin{equation*}
\P_{i,\epsilon}=\{(X,A)\subseteq Q:\phi_i(X,A)\in\P_\epsilon\}.
\end{equation*}
We show that~$\P^*=\bigcup_{i,\epsilon}\P_{i,\epsilon}$, which is~$\Sigma^0_2$.

If~$(X,A)$ has a copy~$(X',A')$ in~$\P_\epsilon$, then there exists~$i$ such that~$\phi_i(X,A)\in\P_{\epsilon/2}$. Indeed, let~$i$ be such that~$\phi_i(X,A)$ is~$\epsilon/2$-homeomorphic to~$(X',A')$ by the previous claim. Every~$\epsilon/2$-deformation of~$\phi_i(X,A)$ is an~$\epsilon$-deformation of~$(X',A')$, so it belongs to~$\P$. Therefore~$\phi_i(X,A)\in\P_{\epsilon/2}$, hence~$(X,A)\in\P_{i,\epsilon/2}$.
\end{proof}

\begin{claim}\label{claim3}
Assume that~$(X,A)$ is~$\P$-minimal and that~$(X,A)\in \P_\epsilon$. For every copy~$(X',A')$ of~$(X,A)$, there exists~$i$ such that~$(X',A')$ is~$\P_{i,\epsilon/2}$-minimal.
\end{claim}

\begin{proof}
Let~$(X',A')$ be a copy of~$(X,A)$. Let~$i$ be such that~$\phi_i(X',A')$ is~$\epsilon/2$-homeomorphic to~$(X,A)$. As in the previous claim,~$(X',A')\in \P_{i,\epsilon/2}$, and we show that~$(X',A')$ is~$\P_{i,\epsilon/2}$-minimal.

If~$(X'',A'')$ is a proper subpair of~$(X',A')$, then~$\phi_i(X'',A'')$ is~$\epsilon/2$-homeomorphic to a proper subpair of~$(X,A)$. The latter is not in~$\P$ as~$(X,A)$ was~$\P$-minimal, and is an~$\epsilon/2$-deformation of~$\phi_i(X'',A'')$, so~$\phi_i(X'',A'')\notin \P_{\epsilon/2}$. In other words.~$(X'',A'')\notin \P_{i,\epsilon/2}$. As a result,~$(X',A')$ is~$\P_{i,\epsilon/2}$-minimal.
\end{proof}

\begin{claim}\label{claim4}
If~$(X,A)$ has strong computable type, then there exists a~$\Pi^0_1$ property~$\P$, a copy~$(X_0,A_0)\subseteq Q$ and an~$\epsilon>0$ such that~$(X_0,A_0)$ is~$\P$-minimal and~$(X_0,A_0)\in\P_\epsilon$. 
\end{claim}
\begin{proof}
Assume that~$(X,A)\subseteq Q$ has strong computable type. Let~$\inj(Q)\subseteq\con(Q)$ be the subspace of \emph{injective} continuous functions~$f:Q\to Q$. It is a~$G_\delta$-set:~$\inj(Q)=\bigcap_{n\in\N}\inj_n(Q)$ where
\begin{equation*}
\inj_n(Q)=\{f\in\con(Q):\forall x,y\in Q,d(x,y)\geq 2^{-n}\implies f(x)\neq f(y)\}
\end{equation*}
is open by Proposition \ref{prop_quant}, because it is obtained as a universal quantification of an open predicate over the compact space~$Q\times Q$ (see Lemma 1.3.10 in \cite{2001vanMill} for a more detailed proof). Being a~$G_\delta$-subset of the Polish space~$\con(Q)$, the space~$\inj(Q)$ is therefore Polish (Theorem 3.11 in \cite{Kechris95}).

Theorem \ref{thm_minimal} implies that for each~$f\in\inj(Q)$ there exists a~$\Pi^0_1$ property~$\P$ (in~$\upVV$) such that~$f(X,A)$ is~$\P$-minimal. There are countably many~$\Pi^0_1$ properties, so by Baire category on~$\inj(Q)$ there exists a~$\Pi^0_1$-property~$\P$, an~$f_0\in\inj(Q)$ and an~$\epsilon>0$ such that for ``almost every'' $f\in B(f_0,\epsilon)$,~$f(X,A)$ is~$\P$-minimal. By ``almost every'', we mean for a dense set of functions~$f\in B(f_0,\epsilon)$. It actually implies that for \emph{every} continuous~$f\in B(f_0,\epsilon)$ the pair~$f(X,A)$ is in~$\P$ because~$\P$ is closed and the injective functions are dense in the continuous ones by Lemma \ref{lem_dense}.

The copy~$f_0(X,A)$ belongs to~$\P_\epsilon$ but may not be~$\P$-minimal. However, we can replace~$f_0$ by any~$f_1\in B(f_0,\epsilon/2)$ such that~$f_1(X,A)$ is~$\P$-minimal, and~$\epsilon$ by~$\epsilon/2$. We then let~$(X_0,A_0)=f_1(X,A)$.
\end{proof}

Finally, we put everything together. Let~$(X_0,A_0)$,~$\P$ and~$\epsilon$ be begin by Claim \ref{claim4}. The~$\Sigma^0_2$ invariant we are looking for is~$\P^*=\bigcup_{i,\epsilon}\P_{i,\epsilon}$. Claim \ref{claim3} implies that every copy of~$(X,A)$ is~$\P_{i,\epsilon}$-minimal for some~$i,\epsilon$.
\end{proof}

\section{Strong computable type: further properties}\label{sec_further}
In this section, we exploit the analysis developed so far to obtain a better structural understanding of strong computable type.

\subsection{A topological necessary condition}
We start with a purely topological understanding of the notion of strong computable type by identifying a topological necessary condition.

\begin{cor}[Necessary condition]\label{cor_necessary}
If a pair~$(X,A)\subseteq Q$ has strong computable type, then there exists~$\epsilon>0$ such that no sequence of~$\epsilon$-deformations of~$(X,A)$ converges to a proper subpair of~$(X,A)$ in the topology~$\upVV$.
\end{cor}
\begin{proof}
From Theorem~\ref{thm_minimal} and the proof of Theorem~\ref{thm_sigma2_equiv},~$(X,A)$ has strong computable type iff there exist a property of pairs~$\P$ which is~$\Pi^0_1$ in~$\upVV$, a copy~$(Y,B)$ of~$(X,A)$ in~$Q$ which is~$\P$-minimal, and an~$\epsilon>0$ such that every~$\epsilon$-deformation of~$(Y,B)$ is in~$\P$.

As~$\P$ is closed in~$\upVV$, every limit of~$\epsilon$-deformations of~$(Y,B)$ in that topology also belongs to~$\P$. As~$(Y,B)$ is~$\P$-minimal, no such limit is a proper subpair of~$(Y,B)$. The same result can be transferred to~$(X,A)$ by using a homeomorphism~$\phi:(X,A)\to (Y,B)$. Indeed, $\phi$ sends proper subpairs of~$(X,A)$ to proper subpairs of~$(Y,B)$ and if~$\delta>0$ is sufficiently small, then~$\phi$ sends~$\delta$-deformations of~$(X,A)$ to~$\epsilon$-deformations of~$(Y,B)$.
\end{proof}

This necessary condition is almost sufficient and captures the purely topological aspects of strong computable type, in the sense that it is equivalent to the following relativization of strong computable type.
\begin{defn}
A pair~$(X,A)$ has \textbf{strong computable type relative to an oracle~$O$} if for every copy~$(Y,B)$ of~$(X,A)$ in~$Q$, there exists a machine with oracle~$O$ computing~$Y$ in~$\V$ from~$(Y,B)$ as an element of~$(\K^2(Q),\upVV)$.
\end{defn}

\begin{cor}\label{cor_sct_relative}
For a pair~$(X,A)$, the following statements are equivalent:
\begin{enumerate}
\item $(X,A)$ has strong computable type relative to some oracle,
\item There exists~$\epsilon>0$ such that no sequence of $\epsilon$-deformations of~$(X,A)$ converges to a proper subpair of~$(X,A)$ in the topology~$\upVV$.
\end{enumerate}
\end{cor}
\begin{proof}
The proof of Corollary \ref{cor_necessary} holds relative to any oracle, which shows (1) $\Rightarrow$ (2). Now assume condition (2). Let~$\P$ be the closure, in~$\upVV$, of the set of~$\epsilon$-deformations of~$(X,A)$. Let~$O$ be an oracle relative to which~$\P$ is~$\Pi^0_1$ in~$\upVV$.  By definition,~$\P$ contains all the~$\epsilon$-deformations of~$(X,A)$ and~$(X,A)$ is~$\P$-minimal by assumption. Therefore, we can apply the argument in the proof of Theorem \ref{thm_minimal}, showing that~$(X,A)$ has strong computable type relative to~$O$.
\end{proof}

The countrapositive of Corollary \ref{cor_necessary} gives a sufficient condition implying that a pair does not have strong computable type. In \cite{AH23}, we prove an effective version which even shows that the pair does not have computable type, as follows.

\begin{thm}
Assume that~$(X,A)\subseteq Q$ is semicomputable and that given~$\epsilon>0$ one can compute~$\delta>0$ such that~$\epsilon$-deformations of~$(X,A)$ converge to a proper subpair~$(Y,B)\subsetneq (X,A)$ such that~$d_H(Y,X)>\delta$. Then~$(X,A)$ does not have computable type.
\end{thm}
Note that neither the~$\epsilon$-deformations nor the limit~$Y$ are required to be computable. The computability assumption is only about the function sending~$\epsilon$ to~$\delta$.

This result therefore identifies a general situation when computable type is equivalent to strong computable type.

\subsection{The \texorpdfstring{$\epsilon$}{epsilon}-surjection property}\label{sec_surjection}
The following special case of Corollary \ref{cor_necessary} has a particularly simple formulation.
\begin{cor}\label{cor_surjection}
If~$(X,A)\subseteq Q$ has strong computable type, then there exists~$\epsilon>0$ such that every continuous function of pairs~$f:(X,A)\to (X,A)$ satisfying~$\rho_X(f,\id_X)<\epsilon$ must be surjective.
\end{cor}
\begin{proof}
Take~$\epsilon$ from Corollary \ref{cor_necessary}. If~$f:(X,A)\to (X,A)$ is an~$\epsilon$-function, then~$f(X,A)$ is an~$\epsilon$-deformation of~$(X,A)$ which is a subpair of~$(X,A)$. By Corollary \ref{cor_necessary} it cannot be a proper subpair, which means that~$f$ must be surjective.
\end{proof}

We do not know whether converse of Corollary \ref{cor_surjection} holds, although we do not expect so. As we discuss now, it holds for particular pairs.

The property from Corollary \ref{cor_surjection} is a slight variation, actually a strengthening, of the following property introduced in \cite{AH22}. A pair~$(X,A)\subseteq Q$ has the $\epsilon$\textbf{-surjection property} if every continuous function~$f:X\to X$ satisfying~$\rho_X(f,\id_X)<\epsilon$ and~$f|_A=\id_A$ is surjective. One of the main results from \cite{AH22} is that for a \textbf{simplicial pair}, i.e.~a pair~$(X,A)$ consisting of a finite simplicial complex~$X$ and a subcomplex~$A$,~$(X,A)$ has computable type if and only if it has the~$\epsilon$-surjection property for some~$\epsilon>0$.

Putting all the results together, we obtain the following characterization for simplicial pairs.
\begin{cor}\label{cor_simplicial}
For a simplicial pair~$(X,A)$, the following statements are equivalent:
\begin{enumerate}
\item $(X,A)$ has strong computable type,
\item $(X,A)$ has computable type,
\item There exists~$\epsilon>0$ such that every continuous function~$f:(X,A)\to (X,A)$ satisfying~$\rho_X(f,\id_X)<\epsilon$ is surjective,
\item There exists~$\epsilon>0$ such that every continuous function~$f:(X,A)\to (X,A)$ satisfying~$\rho_X(f,\id_X)<\epsilon$ and~$f|_A=\id_A$ is surjective.
\end{enumerate}
\end{cor}
\begin{proof}
The following implications hold for any pair:  (1) $\Rightarrow$ (2) and (3) $\Rightarrow$ (4) are straightforward, and (1) $\Rightarrow$ (3) is Corollary \ref{cor_surjection}. The main result from \cite{AH22} is (2) $\Leftrightarrow$ (4). As usual, the proof of the implication (4) $\Rightarrow$ (2) holds relative to any oracle so it actually shows (4) $\Rightarrow$ (1). Therefore we have all the implications.
\end{proof}

\begin{example}\label{ex_surjection}
As already mentioned, we do not know whether the converse of Corollary~\ref{cor_surjection} holds. However, there is a simple example showing that the~$\epsilon$-surjection property does not imply computable type in general. Consider the pair~$(X,A)$ in~$\R^2$ defined as follows:~$X=\bigcup_{i\in\N}X_{i}$
where~$X_{0}=[0,1]\times\{0\}$ and for~$i\geq 1$,~$X_{i}=\{2^{-i}\}\times[0,2^{-i}]$
and~$A=\{(2^{-i},2^{-i}):i\geq 1\}\cup\{(0,0)\}$ (see Figure \ref{fig_surjection}).

\begin{figure}[h]
\centering
\includegraphics{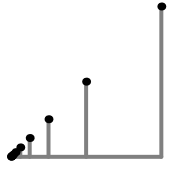}
\caption{The pair~$(X,A)$ in Example \ref{ex_surjection}: it has the surjection property but does not have computable type}\label{fig_surjection}
\end{figure}

Every continuous function~$f:X\to X$ such that~$f|_A=\id_A$ must be surjective, so~$(X,A)$ has the~$\epsilon$-surjection property for every~$\epsilon$. However it does not have computable type: if~$E\subseteq\N$ is a non-computable c.e.~set, then the pair~$(Y,B)$ defined by~$Y=X_0\cup\bigcup_{i\notin E} X_i$ and~$B=\{(2^{-i},2^{-i}):i\geq 1,i\notin E\}\cup\{(0,0)\}$ is a semicomputable copy of~$(X,A)$ but~$Y$ is not computable.
\end{example}

It is an open question whether there is a compact space with infinitely many connected components having strong computable type. The previous results give partial answers to this question. Essentially, such a space cannot have arbitrarily small connected components.

\begin{prop}
Let~$X$ be a compact metric space.

If for every~$\epsilon>0$ there exists a connected component of~$X$ of diameter~$<\epsilon$ which is not a singleton, then~$X$ does not have strong computable type.

If~$X$ contains infinitely many isolated points, then~$X$ does not have strong computable type.
\end{prop}
\begin{proof}
We show that in both situations one can build for each~$\epsilon>0$ a non-surjective continuous~$\epsilon$-function, implying that~$X$ does not have strong computable type by Corollary \ref{cor_surjection}.

Let~$A\subseteq X$ be a connected component of diameter~$<\epsilon$, which is not a singleton. $A$ is the intersection of a decreasing sequence of clopen sets, so by compactness it is contained in a clopen set~$C_\epsilon$ of diameter~$< \epsilon$. The function~$f:X\to X$ defined as the identity on~$X\setminus C_\epsilon$ and sending all~$C_\epsilon$ to some point~$x\in C_\epsilon$ is a continuous~$\epsilon$-function, which is not surjective because~$C_\epsilon$ is not a singleton.

If~$X$ contains infinitely many isolated points then by compactness, they have an accumulation point~$x\in X$. Given~$\epsilon>0$, let~$f:X\to X$ send some isolated point~$y\in B(x,\epsilon)$ to~$x$ and be the identity on~$X\setminus \{y\}$. $f$ is a non-surjective continuous~$\epsilon$-function.
\end{proof}

\subsection{Failing to have computable type in a non-trivial way}

When a set or a pair does not have computable type, it may or may not be easy to build a semicomputable copy which is not computable. For instance, it is straightforward to show that the line segment~$I$ does not have computable type (as a single set, i.e.~without its boundary): if~$r>0$ is a non-computable right-c.e.~real number, then~$[0,r]$ is a semicomputable copy of~$I$ which is not computable.

However, for other sets such as the dunce hat \cite{AH22} or the set shown in Figure \ref{croissant} below, there is no such obvious construction. We can formulate a precise statement expressing this idea, by using the results obtained so far, notably Theorem \ref{thm_minimal}.

\begin{defn}
Let~$X,Y$ be compact metric spaces and~$f:X\to Y$ be continuous. A \textbf{modulus of uniform continuity} for~$f$ is a function~$\mu:\N\to\N$ such that if~$d(x,x')< 2^{-\mu(n)}$ then~$d(f(x),f(x'))\leq 2^{-n}$.
\end{defn}
The choice of strict and non-strict inequalities is not important, but is convenient as it makes the set of functions having modulus~$\mu$ a closed subset of the space~$\con(X,Y)$ of continuous functions from~$X$ to~$Y$.

For instance, a Lipschitz function with Lipschitz constant~$L$ has a modulus of uniform continuity~$\mu(n)=n+\ceil{\log_2(L)}$.

\begin{defn}
Let~$X\subseteq Q$. A copy of~$X$ in~$Q$ is a \textbf{simple copy} of~$X$ if there exists a homeomorphism~$f:X\to Y$ such that both~$f$ and~$f^{-1}$ have a computable modulus of uniform continuity.
\end{defn}

\begin{thm}[Computability of simple copies]\label{thm_nontrivial}
Let~$X\subseteq Q$ be a computable set that does not properly contain a copy of itself. 
Let~$Y$ be a simple copy of~$X$. If~$Y$ is semicomputable then~$Y$ is computable.
\end{thm}

In particular, if~$X$ does not properly contain a copy of itself, then there is no geometrical transformation (scaling, rotation, translation) yielding a semicomputable copy of~$X$ which is not computable, and more generally there is no bilipschitz transformation yielding such a copy.

We need the following result which is folklore, but does not seem to appear in the literature.
\begin{lem}\label{lem_modulus}
Let~$X\subseteq Q$ be a computable compact set and let~$\mu:\N\to\N$ be computable. The set
\begin{equation*}
\K_\mu:=\{f\in\con(X,Q):\mu\text{ is a modulus of uniform continuity of }f\}
\end{equation*}
is effectively compact.
\end{lem}
\begin{proof}
Let~$(x_i)_{i\in\N}$ be a dense computable sequence in~$X$. If~$\mu$ is a computable modulus of continuity of~$f$, then an access to a name of~$f\in\K_\mu$ is computably equivalent to an access to the values of~$f$ on this dense sequence, because using those values and the modulus of continuity, one can computably evaluate~$f$ at any point. In other words, the set~$\K_\mu$ is computably homeomorphic to the set
\begin{equation*}
\L_\mu:=\{(z_i)_{i\in\N}\in Q^\N:\forall i,j,n\in\N,d(x_i,x_j)<2^{-n}\implies d(z_i,z_j)\leq 2^{-n}\}
\end{equation*}
via the map sending~$f\in\K_\mu$ to the sequence~$(f(x_i))_{i\in\N}$. The space~$Q^\N$ is effectively compact and~$\L_\mu$ is a~$\Pi^0_1$-subset of~$Q^\N$, therefore~$\L_\mu$ is effectively compact, and so is~$\K_\mu$.
\end{proof}

We now prove the result.
\begin{proof}[Proof of Theorem \ref{thm_nontrivial}]
Let~$\mu:\N\to\N$ be a computable function. The set~$\J_\mu$ of injective continuous functions~$f:X\to Q$ such that~$\mu$ is a modulus of uniform continuity for both~$f$ and~$f^{-1}$ is an effectively compact subset of~$\con(X,Q)$. Indeed,~$\J_\mu$ is the intersection of~$\K_\mu$ from Lemma \ref{lem_modulus} with the set
\begin{equation*}
\{f\in\con(X,Q):\forall x,x'\in X, d_Q(f(x),f(x'))<2^{-\mu(n)}\implies d_Q(x,x')\leq 2^{-n}\}.
\end{equation*}
As the latter set is~$\Pi^0_1$, its intersection with~$K_\mu$ is effectively compact.

Let~$\P_\mu=\{f(X):f\in \J_\mu\}$. As~$X$ is computable, the function~$\phi:\con(X,Q)\to (\K(Q),\V)$ sending~$f$ to~$f(X)$ is computable. As~$\J_\mu$ is effectively compact, its image~$\P_\mu$ by~$\phi$ is effectively compact as well, so~$\P_\mu$ is a~$\Pi^0_1$-subset of~$\K(Q)$ in the topology~$\V$. As~$X$ does not properly contain a copy of itself, any copy of~$X$ in~$\P_\mu$ is~$\P_\mu$-minimal. Therefore, any semicomputable copy of~$X$ in~$\P_\mu$ is computable, by the same argument as in the proof of Theorem \ref{thm_minimal}. 
\end{proof}

\begin{example}\label{ex_croissant}
The set~$X$ shown in Figure \ref{croissant}, consisting of a disk attached to a pinched torus, does not have computable type. 

It can be proved using the results in \cite{AH22,AH23}. Indeed, the neighborhood of the pinched point is the cone of a graph consisting of two circles attached by a line segment; the line segment is not part of a cycle in the graph, implying that the set does not have computable type. Another way to prove that it does not have computable type is to use the fact that for every~$\epsilon>0$, there is a sequence of~$\epsilon$-deformations of~$X$ converging to a proper subset of~$X$ (see Figure \ref{croissants}). Therefore, it does not have strong computable type by Corollary \ref{cor_necessary}, hence it does not have computable type as it is a finite simplicial complex (Corollary \ref{cor_simplicial}).

As~$X$ does not contain a copy of itself, it has no simple copy which is semicomputable but not computable. It explains why the proof, given in \cite{AH22,AH23}, that such sets do not have computable type is not straightforward.

\begin{figure}[h]
\centering
\subcaptionbox{The set~$X$\label{croissant}}{\includegraphics[scale=.5]{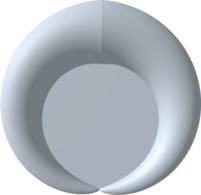}}
\hspace{1cm}
\subcaptionbox{$\epsilon$-deformations converging to a proper subset of~$X$\label{croissants}}{\includegraphics[scale=.5]{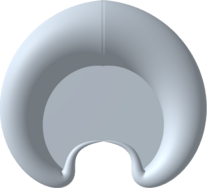}
\includegraphics[scale=.5]{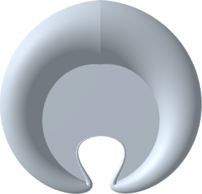}
\includegraphics[scale=.5]{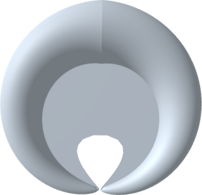}
}
\caption{The set from Example \ref{ex_croissant} and its~$\epsilon$-deformations}
\end{figure}
\end{example}

A similar analysis can be carried out for pairs.

\subsection{Uniformity}

Given a pair~$(X,A)$ which has strong computable type, a natural question
one may ask is whether it is possible to have a single effective procedure that takes any copy~$(Y,B)$ given in the topology~$\upVV$ and computes~$Y$ in the topology~$\V$. It turns out that the answer is almost always negative: uniformity is only possible when~$X$ is a singleton. Therefore, the natural problem is to measure the degree of non-uniformity of this problem, using Weihrauch degrees. It was studied by Pauly \cite{ArnoCCC21} in the case of the circle embedded in~$\R^2$. We give here more results about this question.

\begin{defn}
For a pair~$(X,A)$, let~$\SCT_{(X,A)}$ be the function taking a copy~$(Y,B)$ of~$(X,A)$ in~$\upVV$ and outputting~$Y$ in~$\VV$.
\end{defn}

The first result holds for every pair (which may or may not have strong computable type).
\begin{defn}
\textbf{Closed choice} over~$\N$ is the problem~$\C_\N$ of finding an element in a non-empty set~$A$ of natural numbers, given any enumeration of the complement of~$A$.
\end{defn}

Strong Weihrauch reducibility relative to some oracle is usually denoted by~$\leq_{sW}^t$ (where~$t$ stands for \emph{topological}).

\begin{thm}
Let~$(X,A)$ be a pair such that~$X$ is not a singleton. One has~$\C_\N\leq^t_{sW} \SCT_{(X,A)}$. If~$(X,A)$ has a semicomputable copy then~$\C_{\N}\leq_{sW} \SCT_{(X,A)}$.
\end{thm}

\begin{proof}
Instead of~$\C_\N$, we use the strongly Weihrauch equivalent problem~$\Max$ which sends a non-empty finite subset of~$\N$, represented by its characteristic function, to its maximal element (Theorem 7.13 in \cite{BGP21}). We prove the result when~$(X,A)$ has a semicomputable copy, the general case is obtained by relativizing the argument to an oracle which semicomputes some copy.

Assume that~$(X,A)$ is already embedded as a semicomputable pair in~$Q$. We can assume that~$1/2<\diam X<1$, rescaling~$X$ if needed (and using the assumption that~$X$ is not a singleton).

Given a non-empty finite set~$E\subseteq\N$, let~$m=\max E$. We show how to produce a semicomputable copy~$(X_E,A_E)$ of~$(X,A)$ such that~$2^{-m-1}<\diam{X_E}<2^{-m}$. Given an access to~$X_E$ in the topology~$\V$, one can compute its diameter so one can compute~$m$, showing that~$\Max$ is strongly Weihrauch reducible to~$\SCT_{(X,A)}$.

We now define~$(X_E,A_E)$. It is obtained from~$(X,A)$ by a translation and a scaling by a factor~$2^{-m}$. Given~$E$, one can compute the sequence~$m_i=\max (E\cap[0,i])$ converging to~$m$. The idea is that at stage~$i$, we produce a copy of~$(X,A)$ that is translated and scaled by~$2^{-m_i}$. If at stage~$i+1$, one has~$m_{i+1}=m_i$ then we keep that copy. If~$m_{i+1}>m_i$, then we move to another copy. The new scaling factor is~$2^{-m_{i+1}}$ and we choose the translation so that the new copy is contained in the current~$\upVV$ neighborhood of the previous copy. It is possible because we can assume that at stage~$i$, the~$\upVV$ neighborhood contains a ball of diameter~$2^{-i}$, and~$2^{m_{i+1}}=2^{-i-1}$.
\end{proof}

In the other direction, we can prove that~$\SCT_{(X,A)}\leq_{W}\C_{\N}$ holds when~$(X,A)$ has strong computable type in a particular way. A strong Weihrauch reduction is impossible, because~$\C_\N$ has countably many possible outputs while~$\SCT_{(X,A)}$ has uncountably many.

\begin{thm}\label{thm_choice}Let~$\P$ be a~$\Sigma_{2}^{0}$ invariant. If a pair~$(X,A)$ is~$\P$-minimal then~$\SCT_{(X,A)}\leq_W \C_\N$.
\end{thm}
\begin{proof}
Let~$\P=\bigcup_n \P_n$ where~$\P_n$ are uniformly~$\Pi^0_1$-set in~$\upVV$. Given a copy~$(Y,B)$, let~$E=\{n:(Y,B)\in \P_n\}$. From~$(Y,B)$ in~$\upVV$ one can enumerate the complement of~$E$, and given any~$n\in E$, one can compute~$Y$ by using the fact that~$(Y,B)$ is~$\P_n$-minimal as in the proof of Theorem \ref{thm_minimal}.
\end{proof}
We do not know whether~$\SCT_{(X,A)}$ is always Weihrauch reducible to~$\C_\N$ when~$(X,A)$ has strong computable type. It is left as an open question (Question \ref{que: choice}).

\subsection{Vietoris vs upper Vietoris topology}

We now relate the descriptive complexity of properties of pairs in the Vietoris topology~$\VV$ and upper Vietoris topology~$\upVV$. As~$\upVV$ is weaker than~$\VV$, the descriptive complexity of a property in~$\upVV$ is always an upper bound on its complexity in~$\VV$. Although the properties of a given complexity level are not the same in the topologies~$\V\times\upV$ and~$\upVV$, we show that they actually induce the same class of minimal elements.





\begin{notation}
For a property of pairs~$\P$ let
\begin{equation*}
\ur \P=\{(X,A):\exists (X',A')\subseteq (X,A),(X',A')\in\P\}.
\end{equation*}
\end{notation}

\begin{prop}\label{prop not depend}Let~$\P$ be a property of pairs. The set~$\ur \P$
has the same minimal elements as~$\P$ and 
\begin{enumerate}
\item If~$\P$ is~$\Pi_{1}^{0}$ in~$\VV$, then~$\ur \P$ is~$\Pi_{1}^{0}$
in~$\upVV$,
\item If~$\P$ is~$\Sigma_{2}^{0}$ in~$\VV$, then~$\ur \P$ is~$\Sigma_{2}^{0}$
in~$\upVV$; moreover,~$\ur\P=\bigcup_n\P_n$ where~$\P_n$ are uniformly~$\Pi^0_1$ in~$\upVV$. 
\end{enumerate}
\end{prop}

Therefore, whether a pair~$(X,A)$ is minimal for some~$\Pi_{1}^{0}$ or~$\Sigma_{2}^{0}$ property does not depend on the topology (among~$\VV$ and~$\upVV$) in which the complexity is measured.

In order to prove Proposition \ref{prop not depend}, we need the next result.
\begin{lem}\label{lem_inclusion}
The relation~$\subseteq$, which is the set~$\{(X,X')\in \K^2(Q):X \subseteq X'\}$, is~$\Pi^0_1$ in the topology~$\lowV\times\upV$.
\end{lem}
\begin{proof}
One has~$X\nsubseteq X'$ iff there exists a rational ball~$B$ intersecting~$X$ and such that~$X'$ is contained in~$Q\setminus \cB$. This condition on the pair~$(X,X')$ is an effective open set in the topology~$\lowV \times \upV$. In other words, inclusion is~$\Pi^0_1$ in that topology.
\end{proof}

\begin{proof}[Proof of Proposition \ref{prop not depend}]
It is easy to see that~$\P$ and~$\ur \P$ have the same minimal elements. Now assume that~$\P$ is~$\Pi_{1}^{0}$ in~$\VV$ and let us show that~$\ur P$ is~$\Pi_{1}^{0}$ in~$\upVV$.  One has
\begin{equation*}
(X,A)\in\ur\P\iff \exists (X',A')\in\K(Q)^2, (X',A')\subseteq (X,A)\text{ and }(X',A')\in \P.
\end{equation*}
Using Lemma \ref{lem_inclusion} and the assumption that~$\P$ is~$\Pi^0_1$ in~$\VV$, the relation~$R:=\{(X,A,X',A'):(X',A')\subseteq (X,A)\text{ and }(X',A')\in \P\}$ is~$\Pi^0_1$ in the topology~$\upVV\times\VV$. The space~$(\K^2(Q),\VV)$ is effectively compact, so the projection of~$R$ on its first two components, which is~$\ur\P$, is~$\Pi^0_1$ in the topology~$\upVV$ by Proposition \ref{prop_quant}.

The second item is now easy to prove. If~$\P$ is~$\Sigma_{2}^{0}$ in~$\VV$, then~$\P=\bigcup_{n}\P_n$ where~$\P_n$ are uniformly~$\Pi_{1}^{0}$ in~$\VV$. Therefore,~$\ur \P=\bigcup_{n}\ur \P_n$ and~$\ur \P_n$ are uniformly~$\Pi_{1}^{0}$ in~$\upVV$ by the first statement, so~$\ur \P$ is~$\Sigma_{2}^{0}$ in~$\upVV$.
\end{proof}
%

\subsection{The product and cone spaces}

We give a simple relationship between pairs, their products and their cones, w.r.t.~strong computable type.

\subsubsection{The product of two pairs}

\begin{defn}
The \textbf{product} of two pairs~$(X_1,A_1)$ and~$(X_2,A_2)$ is the pair 
\begin{equation*}
(X_{1},A_{1})\times(X_{2},A_{2}) =(X_{1}\times X_{2},(X_{1}\times A_{2})\cup(A_{1}\times X_{2})).
\end{equation*}
\end{defn}

\begin{prop}
\label{prop: product}If $(X_{1},A_{1})\times(X_{2},A_{2})$
has (strong) computable type and $(X_{2},A_{2})$ has a semicomputable copy
in $Q$, then~$(X_{1},A_{1})$ has (strong) computable type.
\end{prop}

\begin{proof}
Assume that~$(X_{1},A_{1})$ and~$(X_{2},A_{2})$ are embedded in~$Q$ as semicomputable pairs, let us prove that~$X_{1}$ is computable. 

The product~$(X,A)=(X_{1},A_{1})\times(X_{2},A_{2})\subset Q\times Q$ is semicomputable because each pair is. Therefore~$X$ is computable by assumption, so its first projection~$X_1$ is computable. It shows that~$(X_1,A_1)$ has computable type. The proof relativizes, so~$(X_1,A_1)$ has strong computable type.
\end{proof}

If two pairs have strong computable type, it is an open problem whether their product has strong computable type as well.

\subsubsection{The cone of a pair}

\begin{defn}[Cone]\label{def_cone}
 Let~$X$ be a topological space. The \textbf{cone} of~$X$ is the
quotient space~$\cone X=\left(X\times[0,1]\right)/\left(X\times\{0\}\right)$.

Let~$(X,A)$ be a pair. The \textbf{cone pair}~$(Y,B)=\cone{X,A}$
is defined by~$Y=\cone X$ and~$B=X\cup\cone A$, where we consider~$X$ as~$X\times \{1\}$.
\end{defn}

If~$X$ is embedded in~$Q$, then a realization of~$\cone{X}$ in~$Q\cong [0,1]\times Q$ is given by
\begin{equation*}
\cone{X}\cong\{(\lambda,\lambda x):\lambda\in[0,1],x\in X\}.
\end{equation*}

\begin{prop}
Let~$(X,A)$ be a pair. If~$\cone{X,A}$ has (strong) computable type, then~$(X,A)$ has (strong) computable type.
\end{prop}

\begin{proof}
We prove it for computable type and the proof relativizes. Suppose~$(Y,B)=\cone{X,A}$
has computable type and that~$(X_0,A_0)$ is a semicomputable copy of~$(X,A)$ in~$Q$. Let us prove that~$X_0$ is computable. 

The following realization~$(Y_0,B_0)$ of~$\cone{X_0,A_0}$ in~$Q$ is semicomputable:
\begin{align*}
Y_{0}&=\{(\lambda,\lambda x):\lambda\in[0,1],x\in X_0\},\\
B_{0}&=\{(\lambda,\lambda x):\lambda\in[0,1],x\in A_0\}\cup(\{1\}\times X_0).
\end{align*}
The pair~$(Y_{0},B_{0})$ is easily semicomputable so~$Y_0$ is computable by assumption. The intersection~$Z_0$ of~$Y_0$ with~$[1/2,1]\times Q$ is computable as well (an open set~$B$ intersects~$Z_0$ iff the open set~$B\cap ((1/2,1]\times Q)$ intersects~$Y_0$, which is semidecidable). The image of~$Z_0$ under the computable function~$(\lambda,y)\mapsto y/\lambda$ is~$X_0$, which is therefore computable.
\end{proof}
\begin{rem}
The other implication of the previous theorem is false. Let~$L$ be the graph consisting of two circles joined by a line segment, and~$N=\emptyset$. The pair~$(L,N)$ has strong computable type but~$\cone{L,N}$ does not, as shown in~\cite{AH22}, because the graph~$L$ has an edge which belongs to no cycle.

Here is another example. Let~$L$ be the so-called \emph{house with two rooms}, or \emph{Bing's house}, and~$N=\emptyset$. $L$ is contractible and it is proved in \cite{AH22} that it has (strong) computable type. However, the results of \cite{AH22} imply that as~$L$ is contractible, the pair~$(\cone{L},L)$ does not have strong computable type.
\end{rem}

\section{Computability-theoretic properties of ANRs}\label{sec_ANR}
Theorem \ref{thm_sigma2} gives a simple way to prove that a pair~$(X,A)$ has strong computable type by identifying a~$\Sigma^0_2$ topological invariant for which~$(X,A)$ is minimal. In order to apply this result, we need to find suitable~$\Sigma^0_2$ invariants. In this section, we develop results of independent interest that will be applied to define~$\Sigma^0_2$ invariants, expressed in terms of extensions of continuous functions.

The core concept is the notion of Absolute Neighborhood Retract (ANR). It is central in topology and it turns out that it has interesting computability-theoretic properties that we will exploit in order to define~$\Sigma^0_2$ invariants. The general idea is that the space of functions from a space~$X$ to a space~$Y$ cannot be handled in a computable way if~$Y$ is arbitrary; it can if~$Y$ is an ANR.

We recall the definition of Absolute Neighborhood Retracts (ANRs) and explore their computable aspects. The results will be applied in the next section.

\subsection{\label{subsec:Absolute-neighborhood-retracts}Absolute Neighborhood
Retracts (ANRs)}

This important notion was introduced by Borsuk \cite{Borsuk32} and
plays an eminent role in algebraic topology. Moreover, it has very
useful computability-theoretic consequences, which we will take advantage of. We point out that computability-theoretic aspects of compact ANRs have been studied by Collins in \cite{Collins09}, although we do not use these results.
\begin{defn}
Let~$(X,A)$ be a pair. A \textbf{retraction}~$r:X\rightarrow A$
is a continuous function such that~$r|_{A}=\id_{A}$. If a retraction
exists, then we say that~$A$ is a \textbf{retract} of~$X$.
\end{defn}

\begin{defn}
Let~$X$ be a compact space. 
\begin{enumerate}
\item $X$ is an \textbf{absolute retract (AR)} if every copy of~$X$ in~$Q$ is a retract of~$Q$,
\item $X$ is an \textbf{absolute neighborhood retract (ANR)} if every copy~$X'$ of~$X$ in~$Q$ is retract of a neighborhood of~$X'$.
\end{enumerate}
\end{defn}

\begin{fact}
\label{fact_ANR} We recall some classical facts (see \cite{1951some_theorems_on_ANR} and \cite{2001vanMill}).
\begin{enumerate}
\item The~$n$-dimensional ball is an AR, 
\item The~$n$-dimensional sphere is an ANR, 
\item If~$Y$ is an AR and~$(X,A)$ is a pair, then every continuous function~$f:A\to Y$ has a continuous extension~$F:X\to Y$,
\item If~$Y$ is an ANR and~$(X,A)$ is a pair, then every continuous function~$f:A\to Y$ has a continuous extension~$F:U\to Y$ where~$U\subseteq X$ is a neighborhood of~$A$.
\end{enumerate}
\end{fact}

The topologist's sine curve, or the set~$\{0\}\cup\{1/n:n\in\N\}$ are examples of spaces that are not ANRs.

ANRs interact nicely with the notion of homotopy, that we recall now.

\begin{defn}
A \textbf{homotopy} between two continuous functions~$f,g:X\to Y$ is a continuous function~$H:X\times [0,1]\to Y$ such that~$H(x,0)=f(x)$ and~$H(x,1)=g(x)$ for all~$x\in X$. Two functions are \textbf{homotopic} if there exist a homotopy between them. A function is \textbf{null-homotopic} if it is homotopic to a constant function.
\end{defn}

Homotopy between functions is an equivalence relation, whose equivalence classes are called \textbf{homotopy classes}. If two functions to a compact ANR are close to each other, then they are homotopic.
\begin{lem}[Theorem 4.1.1 in \cite{2001vanMill}]\label{lem_hom_ANR}
Let~$Y\subseteq Q$ be a compact ANR. There exists~$\alpha>0$ such that for every space~$X$, if~$f,g:X\to Y$ are continuous and~$\rho_X(f,g)<\alpha$, then~$f,g$ are homotopic.
\end{lem}

Whether a function to an ANR has a continuous extension only depends on the homotopy class of the function. This is Borsuk's homotopy extension theorem (Theorem 1.4.2 in \cite{2001vanMill}).
\begin{thm}[Borsuk's homotopy extension theorem]\label{thm_hom_ext_ANR}
Let~$Y$ be an ANR and~$(X,A)$ be a pair. If~$f,g:A\to Y$ are continuous and~$f$ has a continuous extension~$F:X\to Y$, then every homotopy between~$f$ and~$g$ can be extended to a homotopy between~$F$ and some continuous extension~$G:X\to Y$ of~$g$.
\end{thm}

These results have important computability-theoretic consequences, because arbitrary functions can be replaced by computable functions that are close enough to the original ones so that they are homotopic.

\subsection{Homotopy classes}
If~$X,Y$ are topological spaces and~$f:X\to Y$ is continuous, then we denote by~$[f]$ the homotopy class of~$f$. Let~$[X;Y]$ be the set of homotopy classes of continuous functions from~$X$ to~$Y$.

Let~$Y$ be a fixed compact ANR. Given a compact set~$X\subseteq Q$ in the upper Vietoris topology, we show how the set~$[X;Y]$ can be computed in some way. Note that this question has been addressed, for instance in \cite{cadek_sphere}, when~$X$ and~$Y$ are simplicial complexes presented as combinatorial objects. In our case, the information available about~$X$ and~$Y$ is much more elusive, therefore one might expect that very little can be computed. Note that computable aspects of ANRs have been investigated by Collins \cite{Collins09}, where a compact ANR~$Y$ is presented together with a neighrborhood and a computable retraction. In our case, we do not have access to a retraction. 

In order to state the results, let us discuss the consequences of Borsuk's homotopy extension theorem (Theorem \ref{thm_hom_ext_ANR}). Let~$(X,A)$ be a pair. Whether a continuous function~$f:A\to Y$ has a continuous extension~$F:X\to Y$ does not depend on~$f$, but on its equivalence class in~$[A;Y]$. We denote this class by~$[f]_A$ to make it clear that~$f$ is defined on~$A$ and that we consider homotopies defined on~$A$ (rather than~$X$). If~$f:A\to Y$ and~$F:X\to Y$ are continuous functions, then we say that~$[F]_X$ \textbf{extends}~$[f]_A$ if~$F$ is homotopic to an extension of~$f$. This relation is well-defined, i.e.~does not depend on the representatives of the equivalence classes, because if~$F$ extends~$f$,~$[g]_A=[f]_A$ and~$[G]_X=[F]_X$, then~$G$ is homotopic to an extension of~$g$ by Borsuk's homotopy extension theorem. In particular,~$[F]_X$ extends~$[f]_A$ if and only if~$[F|_A]_A=[f]_A$.

A \textbf{numbering} of a set~$S$ is a surjective partial function~$\nu_S:\subseteq\N\to S$.
\begin{thm}\label{thm_numbering}
Let~$Y$ be a fixed compact ANR. To any compact set~$X\subseteq Q$ we can associate a numbering~$\nu_{[X;Y]}$ of~$[X;Y]$ such that:
\begin{enumerate}
\item $\dom(\nu_{[X;Y]})$ is c.e.~relative to~$X$,
\item Equality~$\{(i,j)\in\dom(\nu_{[X;Y]})^2:\nu_{[X;Y]}(i)=\nu_{[X;Y]}(j)\}$ is c.e.~relative to~$X$,
\item For any pair~$(X,A)$, the extension relation
\begin{equation*}
\{(i,j)\in\dom(\nu_{[X;Y]})\times \dom(\nu_{[A;Y]}):\nu_{[X;Y]}(i)\text{ extends }\nu_{[A;Y]}(j)\}
\end{equation*}
 is c.e.~relative to~$X$ and~$A$,
\item For any pair~$(X,A)$, the extensibility predicate
\begin{equation*}
\{i\in \dom(\nu_{[A;Y]}):\nu_{[A;Y]}(i)\text{ has a continuous extension to }X\}
\end{equation*}
is c.e.~relative to~$X$ and~$A$,
\end{enumerate}
where~$X$ and~$A$ are given as elements of~$(\K(Q),\upV)$.
\end{thm}

Note that the result holds with no computability assumption about~$Y$. We now proceed with the proof of this result.

\begin{lem}\label{lem_ANR}
If~$Y\subseteq Q$ is a compact ANR, then there exist an effective open set~$U$ containing~$Y$, a continuous retraction~$r:U\to Y$ and~$\mu,\alpha>0$ such that:
\begin{enumerate}[(i)]
\item Functions to~$Y$ that are~$\alpha$-close are homotopic,
\item For~$x\in U$,~$d(r(x),x)<\mu$,
\item For~$x,y\in U$,~$d(x,y)<\mu$ implies~$d(r(x),r(y))<\alpha$.
\end{enumerate}
\end{lem}

\begin{proof}
The number~$\alpha$ exists by Lemma \ref{lem_hom_ANR}. As~$Y$ is an ANR, there exists an open set~$W$ containing~$Y$ and a retraction~$r:W\to Y$. Let~$V$ be an open set such that~$Y\subseteq V\subseteq \overline{V}\subseteq W$. As~$\overline{V}$ is compact,~$r$ is uniformly continuous on~$\overline{V}$ so there exists~$\mu>0$ satisfying (iii) for~$x,y\in \overline{V}$. As~$r|_Y=\id_Y$ and~$Y$ is compact, if~$x$ is sufficiently close to~$Y$,~$d(r(x),x)<\mu$. Therefore we can choose~$U\subseteq V$ satisfying (ii). By compactness of~$Y$,~$U$ can be replaced by a finite union of rational balls.
\end{proof}

\begin{proof}[Proof of Theorem \ref{thm_numbering}]
Let~$(f_i)_{i\in\N}$ be a dense computable sequence of functions~$f_i:Q\to Q$. Let~$Y\subseteq Q$ be a compact ANR and let~$U,r,\alpha,\mu$ be provided by Lemma \ref{lem_ANR}. For~$X\subseteq Q$, we define a numbering~$\nu_{[X;Y]}$ of~$[X;Y]$ and prove (1), (4), (2) and (3) in this order. Its domain is
\begin{equation*}
\dom(\nu_{[X;Y]})=\{i\in\N:f_i(X)\subseteq U\},
\end{equation*}
which is clearly c.e.~relative to~$X$, showing (1). For~$i\in\dom(\nu_{[X;Y]})$, we then define~$\nu_{[X;Y]}(i)$ as the equivalence class of~$r\circ f_i|_X:X\to Y$.

We first show that~$\nu_{[X;Y]}$ is surjective. For a continuous function~$f:X\to Y$, let~$\tilde{f}:Q\to Q$ be a continuous extension of~$f$. If~$f_i$ is sufficiently close to~$\tilde{f}$, then~$f_i(X)\subseteq U$ (hence~$i\in\dom(\nu_{[X;Y]})$) and~$r\circ f_i|_X$ is~$\alpha$-close to~$f$. Therefore,~$r\circ f_i|_X$ is homotopic to~$f$ and~$\nu_{[X;Y]}(i)=[f]$.

Let now~$(X,A)$ be a pair in~$Q$ and let~$\nu_{[X;Y]}$ and~$\nu_{[A;Y]}$ be the corresponding numberings. We show condition (4).
\begin{claim}\label{claim_extension}
The set
\begin{equation*}
\{i\in\dom(\nu_{[X;Y]}):\nu_{[A;Y]}(i)\text{ has a continuous extension to }X\}
\end{equation*}
is c.e.~relative to~$X$ and~$A$.
\end{claim}
\begin{proof}[Proof of the claim]
We show that~$r\circ f_i|_A:A\to Y$ has a continuous extension to~$X$ iff there exists~$j\in\dom(\nu_{[X;Y]})$ such that~$\rho_A(f_j,f_i)<\mu$. It will imply the result as this condition is c.e.~relative to~$X$ and~$A$.

If there exists such a~$j$, then~$\rho_A(r\circ f_j,r\circ f_i)<\alpha$ by Lemma \ref{lem_ANR} (iii) so~$r\circ f_j|_A$ and~$r\circ f_i|_A$ are homotopic by Lemma \ref{lem_ANR} (i). As~$r\circ f_j|_A$ obviously has an extension~$r\circ f_j|_X$, so does~$r\circ f_i|_A$ by Borsuk's homotopy extension theorem (Theorem \ref{thm_hom_ext_ANR}).

Conversely, assume that~$r\circ f_i|_A$ has an extension~$f:X\to Y$. One has~$\rho_A(f,f_i)=\rho_A(r\circ f_i,f_i)<\mu$ by Lemma \ref{lem_ANR} (ii) so if~$f_j$ is sufficiently close to~$f$ on~$X$, then~$f_j(X)\subseteq U$ and~$\rho_A(f_j,f_i)<\mu$.
\end{proof}

We now show (2), i.e.~that equality is c.e. 
Let~$\cyl{X}=X\times [0,1]\subseteq Q$ be the cylinder of~$X$ and~$\partial \cyl{X}=X\times \{0,1\}$ its boundary. These sets are semicomputable relative to~$X$. A pair~$(f,g)$ of functions from~$X$ to~$Y$ is nothing else than a function~$h:\partial \cyl{X}\to Y$, defined as~$h(x,0)=f(x)$ and~$h(x,1)=g(x)$. This correspondence is moreover effective, in the sense that there exists a computable function~$\varphi:\dom(\nu_{[X;Y]})^2\to\dom(\nu_{[\partial\cyl{X};Y]})$ such that if~$i,j$ are indices of the homotopy classes of~$f$ and~$g$ respectively, then~$\varphi(i,j)$ is an index of the homotopy class of~$h$. Note that~$f,g$ are homotopic iff~$h$ has a continuous extension to~$\cyl{X}$. Therefore, for~$i,j\in\dom(\nu_{[X;Y]})$, one has~$\nu_{[X;Y]}(i)=\nu_{[X;Y]}(j)$ iff~$\nu_{[\partial\cyl{X};Y]}(\varphi(i,j))$ has an extension to~$\cyl{X}$, which is c.e.~relative to~$X$ by Claim \ref{claim_extension} applied to the pair~$(\cyl{X},\partial\cyl{X})$.

We finally show (3), i.e.~the extension relation is c.e. Note that~$\dom(\nu_{[X;Y]})\subseteq\dom(\nu_{[A;Y]})$ and for~$i\in\dom(\nu_{[X;Y]})$,~$\nu_{[A;Y]}(i)$ is the equivalence class of the restrictions to~$A$ of functions in~$\nu_{[X;Y]}(i)$. For~$i\in\dom(\nu_{[X;Y]})$ and~$j\in\dom(\nu_{[A;Y]})$,~$\nu_{[X;Y]}(i)$ extends~$\nu_{[A;Y]}(j)$ iff~$\nu_{[A;Y]}(i)=\nu_{[A;Y]}(j)$, which is c.e.
\end{proof}

\subsection{Families of \texorpdfstring{$\Sigma^0_2$}{Sigma02} invariants}

The previous results enable us to define families of~$\Sigma^0_2$ invariants. We will study their properties and use particular instances in the next sections.

\subsubsection{Extension of functions}

To each compact ANR~$Y$ we associate a topological invariant~$\E(Y)$. Ultimately, we will focus on~$Y=\SS_n$, the~$n$-dimensional sphere.

\begin{defn}[The invariant $\E(Y)$]\label{def_EY}
Let~$Y$ be a topological space. A pair~$(X,A)$ belongs to~$\E(Y)$ if there exists a continuous function~$f:A\rightarrow Y$ that cannot be extended to a continuous function~$g:X\rightarrow Y$.
\end{defn}

Note that this definition is only interesting when~$A$ is non-empty, otherwise it is never satisfied.

From this definition, the obvious upper bound on the descriptive complexity of~$\E(Y)$ is~$\SSigma^1_2$. However when~$Y$ is a compact ANR, we show that the complexity drops down to~$\Sigma^0_2$, which is much lower and is an effective class, even if there is no computability assumption on~$Y$.

\begin{cor}\label{cor_EANR_sigma}
If~$Y$ is a compact ANR, then~$\E(Y)$
is a~$\Sigma_{2}^{0}$ invariant in~$\upVV$.
\end{cor}

\begin{proof}
It is a simple application of Theorem \ref{thm_numbering}. Indeed, 
\begin{align*}
(X,A)\in\E(Y)&\iff \exists f:A\to Y\text{ having no extension to }X,\\
&\iff \exists i\in\dom(\nu_{[A;Y]}),\forall j\in\dom(\nu_{[X;Y]}),\nu_{[X;Y]}(j)\text{ does not extend }\nu_{[A;Y]}(i)
\end{align*}
which is a~$\Sigma^0_2$ condition.
\end{proof}

Note that although~$\E(Y)$ is a countable union of differences of closed sets, it is never a countable union of closed sets in~$\upVV$ unless it is empty. Indeed, if~$\E(Y)$ is non-empty then it is not an upper set: take~$(X,A)\in\E(Y)$, and observe that~$(X,X)\notin \E(Y)$.

\subsubsection{Null-homotopy of functions}
The previous results provide another family of~$\Sigma^0_2$ invariants.

\begin{defn}[The invariant $\H(Y)$]\label{def_HY}
Let~$Y$ be a topological space. A space~$X$ belongs to~$\H(Y)$ if there exists a continuous function~$f:X\to Y$ which is not null-homotopic, i.e.~not homotopic to a constant function.
\end{defn}

\begin{cor}\label{cor_HY_Sigma2}
If~$Y$ is a compact ANR, then $\H(Y)$ is a~$\Sigma^0_2$ invariant in~$\upVV$.
\end{cor}
\begin{proof}
One has~$X\in\H(Y)$ iff~$[X;Y]$ is non-trivial iff~$\exists i,j\in\dom(\nu_{[X;Y]}), \nu_{[X;Y]}(i)\neq \nu_{[X;Y]}(j)$, which is a~$\Sigma^0_2$ condition. Slightly differently, we fix some~$i_0$ such that~$f_{i_0}:Q\to Q$ is constant with value in~$U$ (implying~$i_0\in\dom(\nu_{[X;Y]})$ for any~$X$), so that~$[X;Y]$ is non-trivial iff~$\exists i\in\dom(\nu_{[X;Y]}),\nu_{[X;Y]}(i)\neq \nu_{[X;Y]}(i_0)$.
\end{proof}

Again,~$\H(Y)$ is not a countable union of closed sets in~$\upV$ unless it is empty, because it is not an upper set in that case: the maximal set~$Q$ is not in~$\H(Y)$ because it is contractible.

The invariant~$\H(Y)$ only applies to single sets. We will define a version for pairs, with~$Y=\SS_n$.

\section{Detailed study of some \texorpdfstring{$\Sigma^0_2$}{Sigma02} invariants}\label{sec_invariants}

Theorem~\ref{thm_sigma2} provides a template to prove that a pair has strong computable type, by identifying a~$\Sigma^0_2$ invariant for which the pair is minimal. One of our goals is to revisit previous results by finding corresponding~$\Sigma^0_2$ invariants when possible, hopefully giving more insight on the problem, providing unifying arguments and sometimes giving new examples for free. 
 We now study in details the~$\Sigma^0_2$ invariants defined in the previous section.

The study of the descriptive complexity of topological invariants is particularly interesting because it lies at the interaction between two branches of topology, namely point-set topology and algebraic topology. Descriptive complexity belongs to point-set topology, because it is about expressing sets of points in terms of open sets; algebraic topology provides topological invariants having low descriptive complexity.

\subsection{Properties of \texorpdfstring{$\E(Y)$}{E(Y)}}
We study the invariant~$\E(Y)$ from Definition \ref{def_EY}, and develop a few techniques to establish whether a pair satisfies this invariant.

First, if~$A$ is a retract of~$X$, then~$(X,A)\notin \E(Y)$. Indeed, every~$f:A\to Y$ has an extension~$f\circ r:X\to Y$, where~$r:X\to A$ is a retraction.

The next result identifies a relation between pairs that preserves the invariant~$\E(Y)$.
\begin{prop}\label{prop_image}
Let~$\phi:(X',A')\to (X,A)$ such that~$\phi|_{A'}:A'\to A$ is a homeomorphism. If~$(X',A')\in\E(Y)$ then~$(X,A)\in\E(Y)$.
\end{prop}
\begin{proof}
We assume that~$(X,A)\notin\E(Y)$ and show that~$(X',A')\notin\E(Y)$. Let~$f:A'\to Y$. The function~$g=f\circ (\phi|_{A'})^{-1}:A \to Y$ has a continuous extension~$G:X\to Y$. The continuous function~$F=G\circ \phi:X'\to Y$ then extends~$f$. Indeed,~$F|_{A'}=G\circ \phi|_{A'}=g\circ \phi|_{A'}=f\circ (\phi|_{A'})^{-1}\circ \phi|_{A'}=f$.
\end{proof}

\paragraph{Converging sequences}
Now we give results showing how~$\E(Y)$ behaves w.r.t.~converging sequences. First, if a sequence of pairs~$(X_i,A_i)$ converges to a pair~$(X,A)$ in the Vietoris topology (i.e., if~$X_i$ converges to~$X$ and~$A_i$ converges to~$A$), then whether all~$(X_i,A_i)$ are (resp.~are not) in~$\E(Y)$ does not imply that~$(X,A)$ is (resp.~is not) in~$\E(Y)$. In other words,~$\E(Y)$ is neither closed nor open in the Vietoris topology. 

Therefore in order to obtain results, we need more complicated assumptions on the way a sequence of pairs converges to a pair.

The first result can be used to prove that a pair is not in~$\E(Y)$ by approximating it with pairs that are not in~$\E(Y)$ in some way. It is taken from the definition of pseudo-cubes in \cite{2020pseudocubes}.
\begin{prop}\label{prop_approx_neg}Let~$(X,A)$ be a pair and~$(X_{i},A_{i})_{i\in\mathbb{N}}$
be a sequence of pairs contained in~$X$ such that
\begin{enumerate}
\item For every~$i$,~$X_{i}\setminus A_{i}$ is open in~$X$ and is contained in~$(X\setminus A)$,
\item For every~$\epsilon>0$, there exists~$i$ such that~$X\setminus(X_{i}\setminus A_{i})\subseteq\Ne{\epsilon}{A}$.
\end{enumerate}
If~$(X_i,A_i)\notin E(Y)$ for all~$i$, then~$(X,A)\notin\E(Y)$.
\end{prop}

\begin{proof}
We assume that~$(X_i,A_i)\notin \E(Y)$ for all~$i$ and show that~$(X,A)\notin\E(Y)$.
 
Let~$f:A\rightarrow Y$ be continuous.~$Y$ is an ANR, so there exists some continuous extension~$\tilde{f}$ of~$f$ to~$\Ne{\epsilon}{A}$, for some~$\epsilon>0$ by Fact~\ref{fact_ANR} (4). Let~$i$ be such that~$X\setminus(X_i\setminus A_{i})\subseteq\Ne{\epsilon}{A}$ by (2), so~$\tilde{f}|_{A_{i}}$ can be defined and has a continuous extension~$F_{i}:X_i\to Y$ by assumption.

We define a continuous extension~$F:X\to Y$ of~$f$, by defining~$F=F_i$ on~$X_i$ and~$F=\tilde{f}$ on~$X\setminus (X_i\setminus A_i)$. Note that~$X_i$ and~$X\setminus(X_{i}\setminus A_{i})$ are closed and their intersection is~$A_{i}$, where these two functions agree. Therefore,~$F$ is well-defined and continuous and extends~$f$. As~$f$ was arbitrary, we have proved that~$(X,A)\notin \E(Y)$.
\end{proof}

The second result can be used to prove that a pair is in~$\E(Y)$ by approximating it with pairs that are in~$\E(Y)$, in a particular way.
\begin{defn}
Let~$Y$ be an ANR and~$X_i$ converge to~$X$ in the upper Vietoris topology. We say that functions~$f_i:X_i\to Y$ are \textbf{asymptotically homotopic} to~$f:X\to Y$ if for some extension~$\tilde{f}:U\to Y$ of~$f$ to a neighborhood~$U$ of~$X$,~$f_i$ is homotopic to~$\tilde{f}|_{X_i}$ for sufficiently large~$i$. 
\end{defn}

There are usually many ways of extending~$f$ to a neighborhood of~$X$. However, any two of them are homotopic on a smaller neighborhood, which implies that the definition does not depend on the choice of an extension of~$f$, and is intrinsic to the functions~$f_i,f$.

\begin{prop}
Whether~$f_i:X_i\to Y$ are asymptotically homotopic to~$f:X\to Y$ does not depend on the choice of a neighborhood~$U$ of~$X$ and an extension~$\tilde{f}:U\to Y$.
\end{prop}
\begin{proof}
Assume that for some~$U$ and~$\tilde{f}:U\to Y$ and sufficiently large~$i$,~$f_i$ is homotopic to~$\tilde{f}|_{X_i}$. Let~$V$ be a neighborhood of~$X$ and~$g:V\to Y$ a continuous extension of~$f$. As~$\tilde{f}$ and~$g$ coincide on~$X$, there exists~$\epsilon>0$ such that their restrictions~$\tilde{f}|_{\Ne{\epsilon}{X}}$ and~$g|_{\Ne\epsilon{X}}$ are as close as needed so that they are homotopic. For sufficiently large~$i$,~$f_i$ is homotopic to~$\tilde{f}|_{X_i}$ by assumption and~$X_i$ is contained in~$\Ne{\epsilon}{X}$, so~$\tilde{f}|_{X_i}$ is homotopic to~$g|_{X_i}$. Therefore,~$f_i$ is homotopic to~$g|_{X_i}$.
\end{proof}

\begin{prop}\label{prop_approx_pos}[Asymptotic homotopy preserves continuous extensibility]
Let~$Y$ be an ANR,~$(X_i,A_i)_{i\in\N}$ and~$(X,A)$ be pairs in~$Q$ such that~$(X_i,A_i)$ converge to~$(X,A)$ in the upper Vietoris topology. Assume that~$f_i:A_i\to Y$ are asymptotically homotopic to~$f:A\to Y$. If~$f$ has a continuous extension on~$X$, then for sufficiently large~$i$,~$f_i$ has a continuous extension on~$X_i$. 
\end{prop}
\begin{proof}
Let~$F:X\to Y$ be a continuous extension of~$f$. It has a continuous extension~$\tilde{F}:U\to Y$ for some neighborhood~$U$ of~$X$. Note that~$U$ is a neighborhood of~$A$ and~$\tilde{F}$ is an extension of~$f$. As~$f_i$ is asymptotically homotopic to~$f$,~$f_i$ is homotopic to~$\tilde{F}|_{A_i}$ for large~$i$. The function~$\tilde{F}|_{A_i}$ has a continuous extension to~$X_i$, namely~$\tilde{F}|_{X_i}$, so~$f_i$ also has a continuous extension as well.
\end{proof}

In particular, when all the involved pairs have the same second component, we obtain the following simple result.
\begin{cor}\label{cor_approx_sameA}
Let~$Y$ be an ANR. Let~$(X_i,A_i)_{i\in\N}$  and~$(X,A)$ be pairs in~$Q$ such that~$A_i=A$ for all~$i$ and~$X_i$ converge to~$X$ in the upper Vietoris topology. If~$f:A\to Y$ has no continuous extension to~$X_i$ for any~$i\in\N$, then~$f$ has no continuous extension to~$X$.
\end{cor}

\subsection{The invariant \texorpdfstring{$\E_n$}{En}}

For~$n\in\N$, the~$n$-dimensional sphere~$\SS_n$ is an ANR. We define~$\E_n=\E(\SS_n)$. This invariant will be used to cover several examples.
\begin{defn}[The invariant $\E_n$]
For~$n\in\N$, a pair~$(X,A)$ is in~$\E_n$ if there exists a continuous function~$f:A\to \SS_n$ that has no continuous extension~$F:X\to\SS_n$.
\end{defn}

Corollary \ref{cor_EANR_sigma} implies that~$\E_n$ is a~$\Sigma^0_2$ invariant, therefore every~$\E_n$-minimal pair has strong computable type, by applying Theorem~\ref{thm_sigma2}.

The simplest example of a pair which is~$\E_n$-minimal is given by the~$n+1$-ball and its bounding~$n$-sphere. We first explain why is it in~$\E_n$. More generally, whether a pair~$(\B_{m+1},\SS_{m})$ satisfies~$\E_n$ depends on the~$m$th homotopy group of the~$n$-sphere.

\begin{prop}\label{prop_ball_sphere_en}
The pair~$(\B_{m+1},\SS_{m})$ is in~$\E_n$ if and only if the homotopy group~$\pi_{m}(\SS_{n})$ is non-trivial. 

In particular,~$(\B_{n+1},\SS_{n})$ is in~$\E_n$ and~$(\B_{m+1},\SS_{m})$ is not in~$\E_n$ for~$m<n$.
\end{prop}

\begin{proof}
As~$\B_{m+1}$ is the cone of~$\SS_m$, a continuous extension of~$f:\SS_m\to\SS_n$ to~$\B_{m+1}$ is nothing else than a null-homotopy of~$f$.
\end{proof}

\subsubsection{Relationship with dimension}
The dimension of a space has a strong impact on the possibility of extending functions to a sphere. First, if~$X$ has dimension at most~$n$, then any continuous function from a closed subset to~$\SS_n$ extends continuously to~$X$.

\begin{thm}[Theorem 3.6.3 in \cite{2001vanMill}]\label{thm_en_dimension}
Let~$n\in\N$. If~$(X,A)$ is a pair with~$\dim(X\setminus A)\leq n$, then~$(X,A)\notin\E_n$.

In particular, if~$\dim(X)\leq n$ then~$(X,A)\notin \E_n$.
\end{thm}

In some cases, if a pair is not in~$\E_n$ then no subpair is in~$\E_n$.
\begin{lem}\label{lem_en_inclusion}
Let~$(X,A)$ be a pair with~$\dim(A)\leq n$. If~$(X,A)\notin \E_n$, then for every~$(Y,B)\subseteq (X,A)$ one has~$(Y,B)\notin \E_n$.
\end{lem}
\begin{proof}
As~$\dim(A)\leq n$, one has~$(A,B)\notin \E_n$. Therefore,~$(X,A)\notin\E_n$ implies~$(Y,B)\notin\E_n$: every function~$f$ on~$B$ has an extension to~$A$ and then to~$X$, whose restriction to~$Y$ is an extension of~$f$.
\end{proof}

This result gives a simple way to prove~$\E_n$-minimality in certain circumstances.
\begin{cor}\label{cor_minimality}
Let~$(X,A)$ be a pair in~$\E_n$, where~$A$ has empty interior and~$\dim(A)\leq n$. If for every~$A\subseteq Y\subsetneq X$ one has~$(Y,A)\notin\E_n$, then~$(X,A)$ is~$\E_n$-minimal.
\end{cor}
\begin{proof}
Let~$(Y,B)$ be a proper subpair of~$(X,A)$. As~$A$ has empty interior,~$Y\cup A$ is a proper subset of~$X$, so~$(Y\cup A,A)\notin \E_n$ by assumption. Therefore,~$(Y,B)\notin \E_n$ by Lemma \ref{lem_en_inclusion}.
\end{proof}

The statement is reminiscent of the discussion about minimality at the end of Section \ref{sec_first_char}. However,~$\E_n$ is not an upper set (indeed,~$(\B_{n+1},\SS_n)\in\E_n$ but~$(\B_{n+1},\B_{n+1})\notin\E_n$) so Lemma \ref{lem_cofinal} does not apply. Indeed, the dimension assumption cannot be dropped in general. The pair~$(\B_4,\SS_3)$ is in~$\E_2$ as witnessed by the Hopf map, but it is not~$\E_2$-minimal, as~$(\B_3,\SS_2)\subsetneq (\B_4,\SS_3)$ is also in~$\E_2$. However, if~$\SS_3\subseteq Y\subsetneq\B_4$ then~$\SS_3$ is a retract of~$Y$ so~$(Y,\SS_3)\notin\ E_2$.

We now apply this technique.
\begin{prop}\label{prop_En_ball}
For every~$n\in\N$, the pair~$(\B_{n+1},\SS_{n})$ is~$\E_n$-minimal.
\end{prop}
\begin{proof}
By Proposition \ref{prop_ball_sphere_en},~$(\B_{n+1},\SS_n)$ is in~$\E_n$. If~$\SS_n\subseteq Y\subsetneq \B_{n+1}$, then~$Y$ retracts to~$\SS_n$ so~$(Y,\SS_n)\notin \E_n$. Therefore by Corollary \ref{cor_minimality}, no proper subpair~$(Y,B)$ of~$(\B_{n+1},\SS_n)$ is in~$\E_n$. As a result,~$(\B_{n+1},\SS_n)$ is~$\E_n$-minimal.
\end{proof}

Therefore, our results give an alternative proof that the pair~$(\B_{n+1},\SS_n)$ has (strong) computable type, which was shown by Miller \cite{Miller02}. We will give more examples of~$\E_n$-minimal pairs in the next sections.

\subsection{The invariant \texorpdfstring{$\H_n$}{Hn}}

It was shown in~\cite{Ilja13,2018manifolds} that closed manifolds and compact manifolds with boundary have computable type. In this section, we identify a~$\Sigma^0_2$ invariant for which manifolds are minimal.

In the~$n$-sphere~$\SS_n$, let~$p$ be an arbitrary point. We recall that two continuous functions of pairs~$f,g:(X,A)\to (\SS_n,p)$ are \textbf{homotopic relative to~$A$} if there exists a homotopy between~$f$ and~$g$ that is constant on~$A$, i.e.~a continuous function~$H:X\times [0,1]\to\SS_n$ such that~$H(x,0)=f(x)$,~$H(x,1)=g(x)$ for all~$x\in X$, and~$H(a,t)=p$ for all~$t\in [0,1]$ and~$a\in A$. A continuous function~$f:(X,A)\to (\SS_n,p)$ is \textbf{null-homotopic relative to~$A$} if~$f$ is homotopic, relative to~$A$, to the constant function~$p$.

\begin{defn}[The invariant $\H_n$]
A pair~$(X,A)$ is in~$\H_{n}$ if there exists a continuous function
of pairs~$f:(X,A)\rightarrow(\SS_{n},p)$ which is not null-homotopic
relative to~$A$.

A space~$X$ is in~$\H_{n}$ if the pair~$(X,\emptyset)$ is in~$\H_{n}$, i.e.~if there exists a continuous function~$f:X\to\SS_n$ which is not null-homotopic.
\end{defn}

Now that for single spaces,~$\H_n$ is precisely~$\H(\SS_n)$ from Definition \ref{def_HY}.

We first prove that~$\H_{n}$ is a~$\Sigma^0_2$ invariant, by reducing it to~$\E_n$. We recall from Definition \ref{def_cone} that for a space~$Y$,~$\cone{Y}$ is the quotient of~$Y\times [0,1]$ obtained by identifying all the pairs~$(y,0)$. A function~$f:Y\to Z$ is null-homotopic if and only if it can be extended to a continuous function~$\tilde{f}:\cone{Y}\to Z$, where~$Y$ is embedded in~$\cone{Y}$ as~$Y\times \{1\}$. A similar equivalence holds for pairs, as follows.

\begin{prop}\label{prop_hn_cone}
For a pair~$(X,A)$, one has
\begin{equation*}
(X,A)\in \H_{n}\iff \cone{X,A}\in\E_n.
\end{equation*}

Moreover, if~$\cone{X,A}$ is~$\E_n$-minimal, then~$(X,A)$ is~$\H_{n}$-minimal.
\end{prop}
\begin{proof}
Let~$(C,D)=\cone{X,A}=(\cone{X},X\cup \cone{A})$. A function~$f:(X,A)\to(\SS_n,p)$ has a canonical extension~$F:D\to\SS_n$ with constant value~$p$ on~$\cone{A}$. Moreover, a null-homotopy of~$f$ relative to~$A$ is nothing else than a continuous extension of~$F$ to~$C$.

Assume that~$(X,A)\in\H_{n}$ and let~$f:(X,A)\to(\SS_n,p)$ be a continuous function which is not null-homotopic relative to~$A$. By the previous observation, the function~$F:D\to\SS_n$ has no continuous extension to~$C$, therefore~$\cone{X,A}\in\E_n$.

Conversely, assume that~$\cone{X,A}\in\E_n$ and let~$G:D\to\SS_n$ have no continuous extension to~$C$. We cannot directly apply the preliminary observation to~$G$ because it is not necessarily constant on~$\cone{A}$. However,~$G$ is homotopic to a function~$F:D\to\SS_n$ which is constant on~$\cone{A}$. Indeed,~$\cone{A}$ is contractible so the restriction of~$G$ to~$\cone{A}$ is null-homotopic. As~$\SS_n$ is an ANR, the homotopy on~$\cone{A}$ can be extended to a homotopy between~$G:D\to\SS_n$ and a function~$F:D\to\SS_n$ which is constant on~$\cone{A}$ (Theorem \ref{thm_hom_ext_ANR}). Let~$p$ be the constant value. As~$G$ has no extension to~$C$, neither does~$F$. We can now apply the preliminary observation: the restriction~$f:X\to\SS_n$ of~$F$ to~$X$ is not null-homotopic relative to~$A$, therefore~$(X,A)\in\H_{n}$.

Assume that~$\cone{X,A}$ is~$\E_{n}$-minimal. If~$(X',A')$ is a proper subpair of~$(X,A)$, then~$\cone{X',A'}$ is a proper subpair of~$\cone{X,A}$ so it does not satisfy~$\E_n$, therefore~$(X',A')\notin\H_{n}$.
\end{proof}

We will see below (Proposition \ref{prop_minimal}) that under additional assumptions, the implication about minimality can be turned into an equivalence.

This characterization allows one to derive results about~$\H_n$ from results about~$\E_n$. We start with the descriptive complexity of~$\H_n$. We already saw in Corollary \ref{cor_HY_Sigma2} that the invariant~$\H(Y)$ of \emph{single} spaces is~$\Sigma^0_2$ in~$\upV$ when~$Y$ is a compact ANR, and a similar result holds for the invariant~$\H_n$ of \emph{pairs}.
\begin{cor}
$\H_{n}$ is a~$\Sigma^0_2$ invariant in~$\upVV$.
\end{cor}
\begin{proof}
Given a pair~$(X,A)\subseteq Q$, a copy~$(C,D)$ of~$\cone{X,A}$ can be defined as follows:
\begin{align*}
C&=\{(t,tx):t\in[0,1],x\in X\},\\
D&=\{(1,x):x\in X\}\cup \{(t,tx):t\in[0,1],x\in A\}.
\end{align*}
The map~$\phi$ sending~$(X,A)\in (\K^2(Q),\upVV)$ to~$(C,D)\in(\K^2(Q),\upVV)$ is easily computable, and~$\H_{n}$ is the preimage of~$\E_n$ by~$\phi$ by Proposition \ref{prop_hn_cone}. As~$\E_n$ is~$\Sigma^0_2$, so is~$\H_{n}$.
\end{proof}

\subsubsection{Relationship with dimension}
As~$\H_n$ can be expressed using~$\E_n$, the relations between~$\E_n$ and the dimension of the space can be similarly formulated for~$\H_n$. Theorem \ref{thm_en_dimension} and Lemma \ref{lem_en_inclusion} have the following consequences.
\begin{prop}\label{prop_hn_dimension}
Let~$(X,A)$ be a pair with~$\dim(X)\leq n-1$. One has~$(X,A)\notin\H_n$.
\end{prop}
\begin{proof}
One has~$\dim(\cone{X})\leq n$ so~$\cone{X,A}\notin \E_n$ by Theorem \ref{thm_en_dimension}, therefore~$(X,A)\notin \H_n$.
\end{proof}

\begin{lem}\label{lem_hn_inclusion}
Let~$(X,A)$ be a pair satisfying~$\dim(X)\leq n$ and~$\dim(A)\leq n-1$. If~$(X,A)\notin \H_n$, then for every pair~$(Y,B)\subseteq (X,A)$ one has~$(Y,B)\notin \H_n$.
\end{lem}
\begin{proof}
Let~$(C,D)=\cone{X,A}$ and~$(E,F)=\cone{Y,B}\subseteq (C,D)$. Reformulating~$\H_n$ in terms of~$\E_n$ thanks to Proposition \ref{prop_hn_cone}, we want to prove that~$(C,D)\notin \E_n$ implies~$(E,F)\notin \E_n$.

It is implied by Lemma \ref{lem_en_inclusion}, which can be applied because~$D=X\cup\cone{A}$ so~$\dim(D)\leq n$.
\end{proof}

Under the same assumptions, we can improve Proposition \ref{prop_hn_cone} to obtain an equivalence.
\begin{prop}\label{prop_minimal}
Let~$(X,A)$ be a pair where~$A$ has empty interior,~$\dim(X)\leq n$ and~$\dim(A)\leq n-1$. One has
\begin{equation*}
(X,A)\text{ is }\H_n\text{-minimal}\iff \cone{X,A}\text{ is }\E_n\text{-minimal}.
\end{equation*}
\end{prop}

Before proving this result, we need to reformulate Borsuk's homotopy extension theorem (Theorem \ref{thm_hom_ext_ANR}). The cylinder of a space~$X$ is~$\cyl{X}=X\times [0,1]$. A homotopy~$H:X\times [0,1]\to Y$ is nothing else than a function from~$\cyl{X}$ to~$Y$, and a null-homotopy is a function from~$\cone{X}$ to~$Y$. Therefore, Borsuk's homotopy extension theorem can be reformulated as follows.
\begin{lem}\label{lem_cyl_cone}
Let~$(X,A)$ be a pair and~$Y$ an ANR. Every continuous function~$f:X\cup \cyl{A}\to Y$ has a continuous extension to~$\cyl{X}$. Every continuous function~$g:\cone{A}\to Y$ has a continuous extension to~$\cone{X}$.
\end{lem}
\begin{proof}
The first statement is Borsuk's homotopy extension theorem. The second one is a particular case for functions that are constant on~$X$. More precisely,~$\cone{A}$ is the quotient of~$X\cup\cyl{A}$ by identifying all the points of~$X$. A function~$g:\cone{A}\to Y$ induces a function~$f:X\cup\cyl{A}\to Y$ which is constant on~$X$, by composing~$g$ with the quotient map. $f$ has a continuous extension~$F:\cyl{X}\to Y$, which is constant on~$X$, and therefore induces a continuous function~$G:\cone{X}\to Y$ extending~$g$. 
\end{proof}

\begin{proof}[Proof of Proposition \ref{prop_minimal}]
Assume that~$(X,A)$ is~$\H_n$-minimal. We need to show that no proper subpair~$(Y,B)$ of~$\cone{X,A}$ is in~$\E_n$. Using Corollary \ref{cor_minimality}, it is sufficient to show the result when~$B=X\cup \cone{A}$. Let then~$Y$ be a proper subset of~$\cone{X}$ containing~$X\cup \cone{A}$. There exists a non-empty open set~$U\subseteq X\setminus A$ and an interval~$J=(a,b)\subseteq [0,1]$ such that~$U\times J$ is disjoint from~$Y$ (we think of~$U\times J$ as embedded in~$\cone{X}$ in the obvious way). Let~$f:X\cup\cone{A}\to \SS_n$. We show how to extend~$f$ to~$\cone{X}\setminus (J\times U)$, giving an extension of~$f$ to~$Y$ by restriction. Figure \ref{fig_extensions} may help visualize the argument.

As~$(X,A)$ is~$\H_n$-minimal by assumption,~$(X\setminus U,A)\notin\H_n$ so~$\cone{X\setminus U,A}\notin\E_n$. Therefore, the restriction of~$f$ to~$(X\setminus U)\cup \cone{A}$ has a continuous extension to~$\cone{X\setminus U}$. So~$f$ has a continuous extension~$g:X\cup \cone{X\setminus U}\to\SS_n$.

Applying Lemma \ref{lem_cyl_cone}, we can extend~$g$ to the quotient of~$X\times [0,a]$ which is homeomorphic to~$\cone{X}$ and to~$\cyl{X}:=X\times [b,1]$. So we obtain a continuous extension~$h$ of~$f$ to the union of~$\cone{X\setminus U}$ with the quotient of~$X\times ([0,a]\cup[b,1])$. That union is precisely~$\cone{X}\setminus (J\times U)$.

\begin{figure}[h]
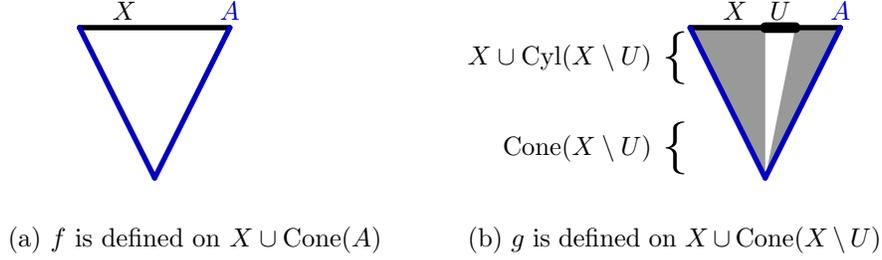
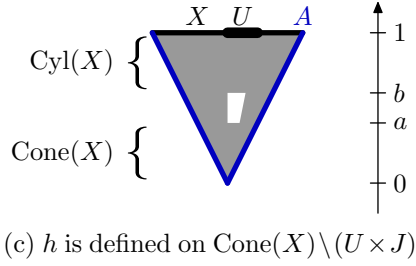

\centering
\subcaptionbox{$f$ is defined on~$X\cup \cone{A}$}[5cm]{\includegraphics{image-1}}\hspace{1cm}
\subcaptionbox{$g$ is defined on~$X\cup \cone{X\setminus U}$}[5.5cm]{\includegraphics{image-2}}
\vspace{1cm}

\subcaptionbox{$h$ is defined on~$\cone{X}\setminus (U\times J)$}[5.5cm]{\includegraphics{image-3}}
\caption{Consecutive extensions of~$f$}\label{fig_extensions}
\end{figure}
\end{proof}

Finally, from the point of view of~$\H_n$, a pair~$(X,A)$ can be replaced by the single space~$X/A$, which is the quotient of~$X$ by~$A$.
\begin{prop}\label{prop_quotient}
For every pair~$(X,A)$,~$(X,A)\in\H_{n}\iff X/A\in\H_{n}$
\end{prop}
\begin{proof}
Let~$p_A\in X/A$ be the equivalence class of~$A$. First observe that~$(X,A)$ can equivalently be replaced by~$(X/A,p_A)$, i.e.~$(X,A)\in\H_{n}\iff (X/A,p_A)\in\H_{n}$. Indeed, there is a one-to-one correspondence between continuous functions~$f:(X,A)\to (\SS_n,p)$ and continuous functions~$g:(X/A,p_A)\to (\SS_n,p)$, and a one-to-one correspondence between null-homotopies of~$f$ relative to~$A$ and null-homotopies of~$g$ relative to~$p_A$.

Now, for a pair~$(Y,q)$, where~$q$ is a point of~$Y$, we show that a function~$g:(Y,q)\to (\SS_n,p)$ is null-homotopic if and only if it is null-homotopic relative to~$q$. Assume that~$g$ is null-homotopic, i.e.~homotopic to a constant function. We can assume w.l.o.g.~that the constant function is~$p$, because all constant functions are homotopic (indeed, the sphere is path-connected). Let~$H_t$ be a homotopy from~$H_0=g$ to~$H_1=p$. Our goal is to define another such homotopy~$H'$ which is constantly~$p$ on~$q$.

We give a different argument for~$n=1$ and~$n\geq 2$.

Let~$n=1$. There exists a continuous function~$\phi:\SS_1\times\SS_1\to\SS_1$ such that~$\phi(x,p)=x$ and~$\phi(x,x)=p$ for all~$x\in\SS_1$. Indeed, identifying~$\SS_1$ with~$\R/\Z$ and~$p$ with the equivalence class of~$0$, let~$\phi(x,y)=x-y\mod 1$. We can then define~$H'_t(y)=\phi(H_t(y),H_t(q))$, it is easy to check that it satisfies the required assumptions.

Let now~$n\geq 2$ (we do not know whether a similar map~$\phi:\SS_n\times\SS_n\to\SS_n$ can be defined but we provide another argument, which cannot be applied to~$n=1$). We want to define a continuous function~$H':Y\times [0,1]\to\SS_n$ extending a certain function~$h':(\{q\}\times [0,1])\cup(Y\times \{0,1\})\to\SS_n$. Let~$h$ be the restriction of~$H$ to that set. We show that~$h'$ is homotopic to~$h$. As~$h$ has a continuous extension, it will imply that~$h'$ has a continuous extension, because~$\SS_n$ is an ANR.

On~$\{q\}\times [0,1]$,~$h$ and~$h'$ define two loops on~$\SS_n$ ($h'$ defines the constant loop). As~$\SS_n$ is simply-connected (indeed,~$n\geq 2$), these two loops are homotopic relative to~$\{q\}\times \{0,1\}$. As~$h$ and~$h'$ coincide on~$Y\times \{0,1\}$, they are homotopic.
\end{proof}

We conclude this section with two simple examples of pairs and sets that are~$\H_n$-minimal.
\begin{prop}
The pair~$(\B_n,\SS_{n-1})$ and the space~$\SS_n$ are~$\H_{n}$-minimal.
\end{prop}
\begin{proof}
It is a direct application of Proposition \ref{prop_hn_cone}: one has~$\cone{\B_n,\SS_{n-1}}=\cone{\SS_n,\emptyset}=(\B_{n+1},\SS_n)$ which is~$\E_n$-minimal by Proposition \ref{prop_En_ball}.
\end{proof}
We will see in the Section \ref{sec_manifold} that, more generally, compact~$n$-dimensional manifolds with or without boundary are all~$\H_n$-minimal.

\subsection{Connections with homology and cohomology}
We end this section with a characterization of~$\E_n$ and~$\H_n$ in terms of homology and cohomology groups for certain spaces. Our goal here is not to give a complete presentation, but to mention intimate connections with algebraic topology, that can be used to study strong computable type. It will be applied to manifolds.

We do not define homology and cohomology groups, but just recall the notations and some classical results. We refer to standard textbooks on algebraic topology for complete expositions of the concepts \cite{HurWal41, Bredon93,Hatcher02}. This section is not essential and may be skipped by the reader who is unfamiliar with algebraic topology.

If~$X$ is a topological space,~$G$ an abelian group and~$n\in\N$, one can define the \textbf{homology groups}~$H_n(X;G)$ and the \textbf{cohomology groups}~$H^n(X;G)$, which are abelian groups. They are topological invariants that, informally, detect the~$n$-dimensional holes of the space~$X$. For instance, when~$X$ is the~$m$-dimensional sphere~$\SS_m$, one has
\begin{equation*}
H_n(\SS_m;G)\cong H^n(\SS_m;G)\cong \begin{cases}
G&\text{if }n=m\text{ or }n=0,\\
0&\text{otherwise},
\end{cases}
\end{equation*}
where~$0$ is the trivial group with one element. 

These groups are also called the \textbf{singular} homology and cohomology groups, to distinguish them from the groups from other homology and cohomology theories, such as the \textbf{\v{C}ech} homology and cohomology groups, denoted by~$\cH_n(X;G)$ and~$\cH^n(X;G)$ respectively. However, these different theories are equivalent for ANRs.
\begin{thm}[Theorem 1 in \cite{Mardesic58}]
If~$X$ is an ANR, then for all~$G$ and~$n$,
\begin{align*}
H_n(X;G)&\cong \cH_n(X;G),\\
H^n(X;G)&\cong \cH^n(X;G).
\end{align*}
\end{thm}

The result applies to compact manifolds with and without boundary, which are ANRs (Corollary A.9 in \cite{Hatcher02}).

\paragraph{Partial characterization of \texorpdfstring{$\E_n$}{En}}
For compact spaces of dimension at most~$n+1$,~$\E_n$ can be characterized in terms of \v{C}ech homology and cohomology. For a pair~$(X,A)$, let~$i:A\to X$ be the inclusion map. Like any continuous map, it induces natural homomorphisms on the homology and cohomology groups, written~$i_*:\cH_n(A;G)\to\cH_n(X;G)$ and~$i^*:\cH^n(X;G)\to\cH^n(A;G)$. It uses the abelian groups~$(\Z,+)$ and~$(\R/\Z,+)$.

\begin{thm}[Corollaries VIII.2 and VIII.3 (p.~149) in \cite{HurWal41}]
Let~$(X,A)$ be a compact pair with~$\dim(X)\leq n+1$. One has
\begin{align*}
(X,A)\in \E_n&\iff i_*:\cH_n(A;\R/\Z)\to \cH_n(X;\R/\Z)\text{ is not injective}\\
&\iff i^*:\cH^n(X;\Z)\to \cH^n(A;\Z)\text{ is not surjective}.
\end{align*}
\end{thm}

The equivalence does not hold in general. For instance, Hopf's fibration is a continuous function~$f:\SS_3\to\SS_2$ which is not null-homotopic, hence does not extend to~$\B_4$, implying that~$(\B_4,\SS_3)\in\E_2$. However, the second homology and cohomology groups of both~$\B_4$ and~$\SS_3$ are trivial, so the only homomorphism between them is bijective.

\paragraph{Partial characterization of \texorpdfstring{$\H_n$}{Hn}}
The invariant~$\H_n$ also admits a characterization in terms of \v{C}ech homology and cohomology, for compact spaces of dimension at most~$n$.
\begin{thm}[Corollary 4 (p.~150) in \cite{HurWal41}]\label{thm_hn_homology}
Let~$X$ be a compact space with~$\dim(X)\leq n$. One has
\begin{align*}
X\in\H_n&\iff \cH_n(X;\R/\Z)\neq 0\\
&\iff \cH^n(X;\Z)\neq 0.
\end{align*}
\end{thm}
Again, the equivalence does not hold in general. Hopf's fibration implies that~$\SS_3\in\H_2$, but~$\cH_2(\SS_3;\R/\Z)\cong \cH^2(\SS_3;\Z)\cong 0$.


Lupini, Melnikov and Nies \cite{LMN23}, and Melnikov and Downey \cite{DM23} with a different proof, proved that if~$X\subseteq Q$ is given in the Vietoris topology~$\V$, then its \v{C}ech cohomology groups~$\check{H}^n(X;\Z)$ can be computably presented (relative to~$X$ and uniformly): one can enumerate a list of generators and the words, i.e.~finite combinations of generators and their inverses, that are equal to the~$0$ element of the group. Their result implies that the non-triviality of a group~$\check{H}^n(X;\Z)$ is a~$\Sigma^0_2$ invariant in the topology~$\V$: the group is non-trivial iff there exists a word which does not equal~$0$, which is a~$\Sigma^0_2$ formula. It seems that their proof even shows that it is a~$\Sigma^0_2$ invariant in the topology~$\upV$. However Proposition \ref{prop not depend} implies that being~$\Sigma^0_2$ in~$\V$ is enough for the purpose of proving strong computable type. More precisely,
\begin{cor}
Let~$X$ be a compact space. If~$X$ is minimal such that~$\check{H}^n(X;\Z)\ncong 0$, then~$X$ has strong computable type.
\end{cor}
\begin{proof}
The invariant~$\I_n=\{X\in\K(Q):\check{H}^n(X;\Z)\ncong 0\}$ is~$\Sigma^0_2$ in the topology~$\V$ by \cite{LMN23} and \cite{DM23}. Therefore,~$\ur{\I_n}$ is~$\Sigma^0_2$ in the topology~$\upV$ and has the same minimal elements as~$\I_n$ (Proposition \ref{prop not depend}), so we can apply Theorem \ref{thm_sigma2}.
\end{proof}

We also mention a stronger connection between maps to the sphere and cohomology. Theorem \ref{thm_hn_homology} is actually a particular case of Hopf's classification theorem, which states that for a compact space of dimension at most~$n$, there is a one-to-one correspondence between homotopy classes of maps to the~$n$-sphere and elements of the~$n$th \v{C}ech cohomology group. It can be found as 
Theorem 2.2 (p.~17) in \cite{Alexandroff47}, reformulated in \cite{Hu48}.
\begin{thm}[Hopf classification theorem]
Let~$X$ be a compact space of dimension~$\leq n$. There is a bijection between~$[X;\SS_n]$ and~$\cH^n(K;\Z)$.
\end{thm}


\section{Examples}\label{sec_examples}

In this section, we illustrate our framework by revisiting previous results stating that a pair has computable type, and showing that it can be explained by their minimality w.r.t.~some suitable~$\Sigma^0_2$ invariant. The advantages of this approach are multiple. A single~$\Sigma^0_2$ invariant can be used for many different sets or pairs, therefore capturing in a unified way the computability-theoretic part of the argument, leaving a purely topological analysis to each specific set or pair, which can often leverage the well-developed field of topology. It makes the study of (strong) computable type more explicitly related to topology.

In particular, most of the arguments in this section are purely topological.

\subsection{Examples of \texorpdfstring{$\E_n$}{En}-minimal pairs}

In this section we show how~$\E_n$-minimality covers many of the known examples of pairs having computable type. We also give a new example to illustrate its scope.

\subsubsection{Chainable continuum between two points}

We revisit a result due to Iljazovi{\'c} in \cite{Ilja09}. Let us first recall the definitions. A continuum is a connected compact metric space.


\begin{defn}\label{def_chain}
Let~$X$ be a metric space. A finite sequence~$\mathcal{C}$ of
nonempty open subsets~$C_{0},\ldots,C_{m}$ of~$X$ is said to be
a \textbf{chain} if~$C_{i}\cap C_{j}\neq\emptyset\Leftrightarrow|i-j|\leq1$,~$\forall i,j\in\{0,\ldots,m\}$.
A chain~$\mathcal{C}=\{C_{0},\ldots,C_{m}\}$ \textbf{covers}~$X$ if~$X\subseteq\bigcup_{i}C_{i}$;
for~$\epsilon>0$, it is an~$\epsilon$\textbf{-chain} if~$\max_{i}(\diam{C_{i}})<\epsilon$.

For~$a,b\in X$,~$X$ is \textbf{chainable from~$a$ to~$b$} if for every~$\epsilon>0$, there exists an~$\epsilon$-chain~$\mathcal{C}=\{C_{0},\ldots,C_{m}\}$ in~$X$ which covers~$X$ and such that~$a\in C_{0}$ and~$b\in C_{m}$.
\end{defn}

Iljazovi\'c proved in \cite{Ilja09} that if~$X$ is chainable from~$a$ to~$b$, then the pair~$(X,\{a,b\})$ has computable type. This result and its proof can be reformulated as the~$\E_0$-minimality of the pair.
\begin{prop}
If~$X$ is compact connected and is chainable from~$a$ to~$b$, then the pair~$(X,\{a,b\})$ is~$E_{0}$-minimal.
\end{prop}

\begin{proof}
Take~$\SS_{0}=\{a,b\}$ and~$f=\id_{\SS_{0}}:\SS_0\to\SS_0$. As~$X$ is connected, the function~$f$ has no continuous extension~$g:X\rightarrow\SS_{0}$, because~$g^{-1}(a)$ and~$g^{-1}(b)$ would disconnect~$X$. Hence,~$(X,\{a,b\})$ is in~$E_{0}$.

Let~$x\in X\setminus \{a,b\}$ and~$\epsilon>0$ be such that~$B(x,\epsilon)\subseteq X\setminus \{a,b\}$. Let~$f:\{a,b\}\to\SS_0$. The points~$a,b$ must be disconnected in~$X\setminus B(x,\epsilon)$. Indeed, let~$\mathcal{C}=\{C_{0},\ldots,C_{m}\}$ be an~$\epsilon$-chain covering~$X$ with~$a\in C_{0}$ and~$b\in C_{m}$. Let~$k$ be such that~$x\in C_k$. As~$\diam{C_k}<\epsilon$,~$a\notin C_k,b\notin C_k$ and~$C_k\subseteq B(x,\epsilon)$. Therefore~$X\setminus B(x,\epsilon)$ is covered by the two disjoint open sets~$U_a:=\bigcup_{i<k}C_i$ and~$U_b:=\bigcup_{i>k}C_i$, containing~$a$ and~$b$ respectively.

The disconnection can be turned into a continuous extension~$g:X\setminus B(x,\epsilon)\to \SS_0$, sending~$U_a$ to~$f(a)$ and~$U_b$ to~$f(b)$. Therefore,~$(X\setminus B(x,\epsilon),\{a,b\})\notin \E_0$. It implies that~$(X,\{a,b\})$ is~$\E_0$-minimal because any proper subset of~$X$ is contained in such a~$X\setminus B(x,\epsilon)$.
\end{proof}

We will cover other examples from \cite{Ilja09} in Section \ref{sec_circ_chain}, using another~$\Sigma^0_2$ invariant.
%
\subsubsection{Pseudo \texorpdfstring{$n$}{n}-cubes}

In \cite{2020pseudocubes}, the authors introduce the notion of a pseudo-$n$-cube, which is a pair~$(X,A)$ which can somehow be approximated by~$n$-dimensional cubes and their boundaries. They prove that pseudo-$n$-cubes have computable type.

We show how the result can be proved using the invariant~$\E_{n-1}$, by showing how the~$\E_{n-1}$-minimality of the~$n$-dimensional cube implies the~$\E_{n-1}$-minimality of pseudo-$n$-cubes.

We first recall the definition of pseudo-$n$-cubes from \cite{2020pseudocubes}. Fix some~$n\in\N^{*}$ and let~$S^0_{1},\ldots,S^0_{n}$, $S^1_1,\ldots,S^1_n$ be the faces of the~$n$-cube~$I^{n}$ defined by
\begin{align*}
S^0_{i} & =\{(x_{1},\ldots,x_{n})\in I^{n}:x_{i}=0\},\\
S^1_{i} & =\{(x_{1},\ldots,x_{n})\in I^{n}:x_{i}=1\},
\end{align*}
where~$i\in\{1,\ldots,n\}$. Let~$S=\bigcup_{i}S^0_{i}\cup S^1_i$ be the boundary of~$I^n$. 

The Hausdorff distance between two non-empty compact subsets~$A,B$ of a metric space~$(X,d)$ is
\begin{equation*}
\d H{A}{B}=\max\left(\max_{x\in A}\min_{y\in B}d(x,y),\max_{y\in B}\min_{x\in A}d(x,y)\right).
\end{equation*}

\begin{defn}
A compact pair~$(X,A)$ in~$\R^{n}$ is called a \textbf{pseudo-$n$-cube}
with respect to some finite sequence of compact subsets~$A^0_{i}$ and~$A^1_i$,~$1\leq i\leq n$,
if~$A=\bigcup_{i}A^0_{i}\cup A^1_i$,~$A^0_{i}\cap A^1_{i}=\emptyset$ and for
each~$\epsilon>0$ there exists a continuous injection~$f:I^{n}\rightarrow\R^{n}$
such that
\begin{enumerate}
\item $f(I^{n})\subseteq X$,
\item $f(\int{\R^{n}}{I^{n}})\cap A=\emptyset$,
\item $X\setminus f(\int{\R^{n}}{I^{n}})\subseteq\Ne{\epsilon}{f(S)}$,
\item $\d H{f(S^0_{i})}{A^0_{i}}<\epsilon$ and $\d H{f(S^1_{i})}{A^1_{i}}<\epsilon$.
\end{enumerate}
\end{defn}

We now show how the proof that pseudo-$n$-cubes have computable type can be simplified by showing that they are~$\E_{n-1}$-minimal. 
\begin{prop}
Every pseudo-$n$-cube is~$\E_{n-1}$-minimal.
\end{prop}

\begin{proof}
A pseudo-cube is approximated by cubes. We use Propositions \ref{prop_approx_pos} and \ref{prop_approx_neg} to derive~$\E_{n-1}$-minimality of pseudo-$n$-cubes from~$\E_{n-1}$-minimality of~$n$-cubes (Proposition \ref{prop_En_ball}).

We recall that the pair~$(I^n,S)\cong (\B_n,\SS_{n-1})$ is in~$\E_{n-1}$ because the identity on~$S$ has no continuous extension to~$I^n$. A similar function~$h:A\to S$ can be defined that has no extension to~$X$. The function~$h$ satisfies~$h(A^0_i)\subseteq S^0_i$ and~$h(A^1_i)\subseteq S^1_i$ and is built as follows. On an intersection~$A^{b_1}_1\cap\ldots \cap A^{b_n}_n$ (where~$b_i\in\{0,1\}$),~$h$ sends every point to the vertex~$(b_1,\ldots,b_n)$ of~$I^n$. We then progressively extend~$h$, sending intersections of~$n-1$ sets to segments, intersections of~$n-2$ sets to squares, then cubes, etc.~and finally sending single sets~$A^0_i$ and~$A^1_i$ to the~$n-1$-cubes which are the faces~$S^0_i$ and~$S^1_i$ of~$I^n$. All the extensions can be done using the Tietze extension theorem.

Once~$h:A\to S$ is defined, it has a continuous extension~$\tilde{h}$ to a neighborhood of~$A$ because~$S$ is an ANR. For sufficiently small~$\epsilon$ and a cube~$(I^n,S)$ embedded in~$X$ as in the definition,~$\tilde{h}$ is defined on~$S$ and sends~$S^b_i$ close to~$S^b_i$, in particular disjoint from~$S^{1-b}_i$. Thus,~$\tilde{h}|_S:S\to S$ is homotopic to~$\id_S$. In other words, the functions~$\id_S$ on the copies of~$(I^n,S)$ are asymptotically homotopic to~$h:A\to S$. As~$\id_S$ has no continuous extension to~$I^n$,~$h$ has no continuous extension to~$X$ by Proposition \ref{prop_approx_pos}. As a result,~$(X,A)\in \E_{n-1}$.

We now prove minimality. By definition of pseudo-$n$-cubes, for each~$\epsilon$ there exists a copy~$(X_i,A_i)$ of the~$n$-cube in~$X$ satisfying the conditions of Proposition \ref{prop_approx_neg}: 
\begin{itemize}
\item $X_i\setminus A_i$ is open in~$X$ (indeed, it is the image by a continuous injection map of an open subset of~$\R^n$ in~$X\subseteq\R^n$, so it is open by invariance of domain) and contained in~$X\setminus A$,
\item $X\setminus (X_i\setminus A_i)\subseteq\Ne{\epsilon}{A}$.
\end{itemize}
If~$U\subseteq X\setminus A$ is an open set intersecting~$X$, then the same conditions are satisfied by the pairs~$(X_i\setminus U,A_i\setminus U)$ and~$(X\setminus U,A)$. For sufficiently large~$i$,~$U$ intersects~$X_i$ so~$(X_i\setminus U,A_i\setminus U)\notin \E_{n-1}$ by~$\E_{n-1}$-minimality of~$(X_i,A_i)$. We can apply Proposition \ref{prop_approx_neg}, which implies that~$(X\setminus U,A)\notin \E_{n-1}$. Therefore,~$(X,A)$ is~$\E_{n-1}$-minimal (we again use Corollary \ref{cor_minimality} which can be applied, as~$A\subseteq\R^n$ has empty interior, therefore has dimension at most~$n-1$ by Theorem 3.7.1 in \cite{2001vanMill}).
\end{proof}
%

\subsubsection{\label{par:More-examples}A new example}

In this section we illustrate how~$\E_n$-minimality can be used to obtain new examples of spaces having computable type property, with little effort. We present an example, but many variations are possible.
\begin{defn}
The \textbf{topologist\textquoteright s sine curve}~$T$, or \textbf{Warsaw
sine curve}, is defined by
\begin{equation*}
T=\left\{ \left(x,\sin \tfrac{1}{x}\right):0<x\leq 1\right\} \cup\left\{ (0,y): -1\leq y\leq 1\right\} .
\end{equation*}

We define the \textbf{Warsaw saucer}~$S$ by
\begin{equation*}
S =\left\{(x,y,z)\in\R^{3}:0<x^{2}+y^{2}\leq 1\text{ and }z=\sin\tfrac{1}{\sqrt{x^{2}+y^{2}}}\right\}\cup \{(0,0,z):-1\leq z\leq1\}.
\end{equation*}

We define its boundary~$\partial S$ as its bounding circle
\begin{equation*}
\partial S=\left\{ (x,y,\sin(1))\in\R^{3}:x^{2}+y^{2}=1\right\} \cong\SS_{1}.
\end{equation*}
\end{defn}

It is obtained from the closed topologist's sine curve, making it rotate around the vertical segment, see Figure \ref{fig_surface}.
\begin{figure}[h]
\subcaptionbox{ The Warsaw saucer~$S$}{\includegraphics[scale=0.25]{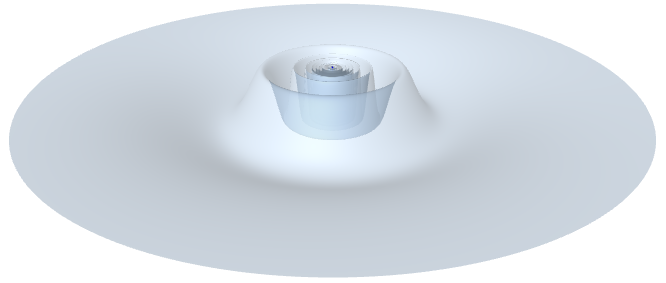}}\hspace{1cm}
\subcaptionbox{A half-cut of~$S$}{\includegraphics[scale=0.25]{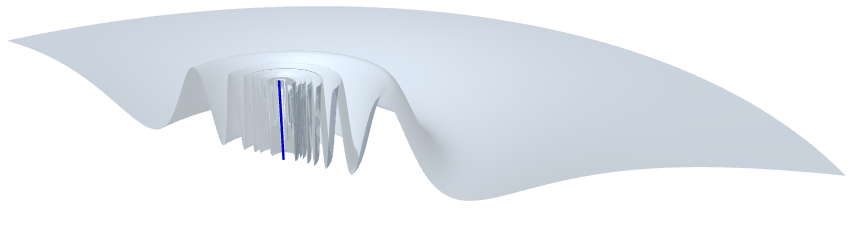}}
\caption{}
\label{fig_surface}
\end{figure}

The set~$S$ is compact and connected but neither locally connected
nor path-connected. It is neither a manifold, nor a simplicial complex, nor a pseudo-$n$-cube so none of the previous results applies to~$(S,\partial S)$. However, it can easily be proved to have strong computable type, because it is closely related to the pair~$(\B_2,\SS_1)$ and is therefore~$\E_1$-minimal.
\begin{prop}
The pair~$(S,\partial S)$ is~$\E_1$-minimal and hence has strong computable
type.
\end{prop}

\begin{proof}
We show the~$\E_1$-minimality of~$(\B_2,\SS_1)$ implies the~$\E_1$-minimality of~$(S,\partial S)$.

The surface~$S$ can be approximated by copies of the disk~$\B_2$ that share the same boundary~$\SS_1$. Indeed, let
\begin{equation*}
X_i=\left\{(x,y,z):\frac{1}{i\pi}\leq \sqrt{x^2+y^2}\leq 1\text{ and }z=\sin\tfrac{1}{\sqrt{x^2+y^2}}\right\}\cup\left\{(x,y,0):\sqrt{x^2+y^2}\leq \frac{1}{i\pi}\right\}.
\end{equation*}

The sets~$X_i$ converge to~$S$, and they all have the same boundary~$\SS_1$. We know that~$(X_i,\SS_1)=(\B_2,\SS_1)\in\E_1$, moreover the same function, which is the identity~$\id:\SS_1\to\SS_1$, has no continuous extension to~$X_i$. Therefore, this function has no continuous extension to~$S$ by Corollary \ref{cor_approx_sameA}.

We now prove~$\E_1$-minimality. The projection~$\pi$ on the horizontal plane sends~$S$ to~$\B_2$ and is the identity on~$\partial S=\SS_1$. If~$X$ is a proper compact subset of~$S$ containing~$\partial S$, then~$\pi(X)$ is a proper compact subset of~$\B_2$. As~$(\B_2,\SS_1)$ is~$\E_1$-minimal,~$(\pi(X),\SS_1)\notin\E_1$ so~$(X,\partial S)\notin\E_1$ by Proposition \ref{prop_image} (more directly, every~$f:\SS_1\to\SS_1$ has a continuous extension~$F:\pi(X)\to\SS_1$, inducing a continuous extension~$F'=F\circ \pi:X\to \partial S$ of~$f$ to~$X$). As a result,~$(S,\partial S)$ is~$\E_1$-minimal.  
\end{proof}

\subsection{Compact \texorpdfstring{$n$}{n}-manifolds are \texorpdfstring{$\H_n$}{Hn}-minimal}\label{sec_manifold}
It was proved in \cite{Ilja13} and \cite{2018manifolds} that compact manifolds with or without boundary have computable type, as pairs and single spaces respectively. In this section, we show how this result can be derived from classical results in algebraic topology implying that $n$-dimensional manifolds are~$H_n$-minimal, and applying Theorem \ref{thm_sigma2}. We first recall the definition of a manifold.

\begin{defn}
A \textbf{manifold} of dimension~$n$, or more concisely an~\textbf{$n$-manifold},
is a Hausdorff space~$M$ in which each point has an open neighborhood homeomorphic to~$\R^{n}$.

An~\textbf{$n$-manifold with boundary} is a Hausdorff space~$M$
in which each point has an open neighborhood homeomorphic either to~$\R^{n}$
or to the half-space~$\R_{+}^{n}=\{(x_{1},\ldots,x_{n})\in\R^{n}:x_{n}\geq0\}$ by a homeomorphism which maps~$x$ to a point~$(x_1,\ldots,x_{n-1},0)$.
The set of points which have an open neighborhood homeomorphic to
the half-space is called the \textbf{boundary} of~$M$ and is denoted by~$\partial M$. 
\end{defn}

The next result seems to be folklore but is not stated in any standard textbook on algebraic topology. However, it can be derived from classical results about homology of manifolds. 

\begin{thm}\label{thm_manifold_hn}
Every connected compact~$n$-manifold~$M$ is~$\H_{n}$-minimal. Every connected compact~$n$-manifold with boundary~$(M,\partial M)$ is~$\H_n$-minimal.
\end{thm}

We give here the proof for manifolds without boundary, and put the proof for manifolds with boundary in the appendix (it is similar but requires more technical results about pairs).

\begin{proof}
The result is a consequence of Theorem \ref{thm_hn_homology}, which reduces~$\H_n$ to homology groups, and of classical results about homology groups of manifolds. 

Let us use the abelian group~$G=\R/\Z$. As it contains an element of order~$2$, namely~$1/2$, Corollary 7.12 in \cite{Bredon93} implies that if~$M$ is a connected $n$-manifold, then~$H_n(M;G)\neq 0$ if and only if~$M$ is compact. Therefore, if~$M$ is a compact connected~$n$-manifold and~$x\in M$, then
\begin{align*}
H_n(M;G)&\ncong 0,\\
H_n(M\setminus \{x\};G)&\cong0,
\end{align*}
because~$M\setminus\{x\}$ is a non-compact connected manifold.

If~$x\in M$ and~$B$ is an open Euclidean ball around~$x$, then~$M\setminus \{x\}$ \emph{deformation retracts} to~$M\setminus B$, which means that there is a retraction~$r:M\setminus\{x\}\to M\setminus B$ such that if~$i:M\setminus B\to M\setminus\{x\}$ is the inclusion map, then~$i\circ r$ is homotopic to~$\id_{M\setminus \{x\}}$. It is a classical result that deformation retractions preserve homology groups, so~$H_n(M\setminus B)\cong H_n(M\setminus \{x\})\cong 0$.

Therefore, Theorem \ref{thm_hn_homology} implies that~$M\in\H_n$ and~$M\setminus B\notin\H_n$. It implies that~$M$ is~$\H_n$-minimal by Lemma \ref{lem_hn_inclusion}: if~$X$ is a proper subset of~$M$, then~$X\subseteq M\setminus B$ for some~$B$, so~$X\notin \H_n$.
\end{proof}

Theorem \ref{thm_manifold_hn} implies the result in \cite{Ilja13} that connected compact manifold with or without boundary have (strong) computable type. However the proof is inherently different and has new implications. For instance, it implies that~$\SCT_{M,\partial M}\leq_W\C_\N$ by Theorem \ref{thm_choice}. Note that a disconnected compact manifold also has strong computable type because it is a finite union of connected compact manifolds, but it is not~$\H_n$-minimal because every connected component is in~$\H_n$.

Moreover, Theorem \ref{thm_manifold_hn} immediately implies that other pairs have strong computable type, as follows.

\begin{cor}
If~$M$ is a compact connected~$n$-manifold with possibly empty boundary~$\partial M$, then the pair
\begin{equation*}
\cone{M,\partial M}=(\cone{M},M\cup\cone{\partial M})
\end{equation*}
is~$\E_n$-minimal hence has strong computable type.
\end{cor}
\begin{proof}
All the assumptions of Proposition \ref{prop_minimal} are met:~$\partial M$ has empty interior in~$M$,~$\dim(M)=n$ and~$\dim(\partial M)=n-1$. As~$(M,\partial M)$ is~$\H_n$-minimal by Theorem \ref{thm_manifold_hn},~$\cone{M,\partial M}$ is~$\E_n$-minimal.
\end{proof}
The cone of a manifold~$M$ is a manifold only when~$M$ is a sphere of a ball, so this result is indeed new.


\subsection{Circularly chainable continuum which are not chainable}\label{sec_circ_chain}

In~\cite{Ilja09}, it is proved that every compact metrizable space which is circularly chainable but not chainable has computable type. The simplest examples of such sets are given by the circle and the Warsaw circle. We briefly show how the proof of this result can be reformulated in our framework by finding a suitable~$\Sigma^0_2$ invariant.

We already saw the notion of chain in Definition \ref{def_chain}. We recall other related notions, taken from~\cite{Ilja09}. Let~$(X,d)$ be a metric space. A finite sequence~$C=(C_{0},\ldots,C_{n})$ of non-empty open subsets of~$X$ is called:
\begin{itemize}
\item A \textbf{chain} if for all~$i,j\in\{0,\ldots,n\}$,~$C_{i}\cap C_{j}\neq\emptyset\iff|i-j|\leq1$,
\item A \textbf{circular chain} if for all~$i,j\in\{0,\ldots,n\}$,~$C_{i}\cap C_{j}\neq\emptyset\iff (|i-j|\leq1$ or~$\{i,j\}=\{0,n\})$. 
\item A \textbf{quasi-chain} if for all~$i,j\in\{0,\ldots,n\}$,~$|i-j|>1\implies C_{i}\cap C_{j}=\emptyset$.
\end{itemize}
Its mesh is defined by~$\M C=\max_{0\leq i\leq n}(\diam {C_i})$. It covers a set~$S\subseteq X$ if~$S\subseteq\bigcup_{i}C_{i}$. 

A set~$S\subseteq X$ is \textbf{chainable} (resp.~\textbf{circularly chainable}, \textbf{quasi-chainable}) if for every~$\epsilon>0$ there exists a chain (resp.~circular chain, quasi-chain) of mesh~$<\epsilon$ covering~$S$.

If~$S$ is connected, then~$S$ is chainable if and only if~$S$ is quasi-chainable (Lemma 30 in \cite{Ilja09}).

\begin{prop}
Not being quasi-chainable is a~$\Sigma_{2}^{0}$ invariant in~$\upV$. If a compact space is circularly chainable but not chainable, then it is minimal satisfying this invariant and hence it has strong computable type.
\end{prop}
\begin{proof}
If~$U\subseteq Q$ is a finite union of rational open balls, then we write~$\overline{U}$ for the corresponding union of closed balls. By a standard compactness argument, a compact set~$S\subseteq Q$ is quasi-chainable if for every rational~$\epsilon>0$ there exists a quasi-chain~$C=(C_0,\ldots,C_n)$ of mesh~$<\epsilon$ covering~$S$, with the following additional properties:~$\overline{C_i}\cap \overline{C_j}=\emptyset$ for~$|i-j|>1$, and~$\diam{\overline{C_i}}<\epsilon$. Finite unions of rational closed balls are effectively compact, so these additional properties are c.e., and whether~$C$ covers~$S$ is effectively open in the upper Vietoris topology. Therefore, being quasi-chainable is a~$\Pi^0_2$ invariant.

If~$S$ is not chainable then~$S$ is not quasi-chainable. If in addition~$S$ is circularly chainable, then every proper compact subset of~$S$ is quasi-chainable. Indeed, let~$T\subsetneq S$ be compact,~$x\in S\setminus T$ and~$\epsilon<d_Q(x,T)$, and let~$C=(C_0,\ldots,C_n)$ be a circular chain of mesh~$<\epsilon$ covering~$S$. We can assume that~$x\in C_n$, applying a circular permutation if needed. As~$\diam{C_n}<\epsilon<d_Q(x,T)$, the set~$C_n$ is disjoint from~$T$. As a result,~$(C_0,\ldots,C_{n-1})$ is a quasi-chain covering~$T$.
\end{proof}

It can be shown that if a compact space is in~$\H_1$, then it is not quasi-chainable. We do not know whether the converse implication holds, in particular whether spaces that are circularly chainable but not chainable are~$\H_1$-minimal.

\subsection{A space that properly contains copies of itself}\label{sec_infinite}

Theorem \ref{thm_sigma2} states in particular that if a compact space~$X$ is minimal for some~$\Sigma^0_2$ topological invariant then~$X$ has strong computable type. We show here that the converse implication does not hold. We actually prove more: we build a compact space~$X$ that has strong computable type and contains a proper subset~$X'$ that is homeomorphic to~$X$. It implies that~$X$ cannot be minimal for \emph{any} topological invariant, because any invariant satisfied by~$X$ is also satisfied by~$X'$. Moreover, the space~$X$ is to our knowledge the first example of an infinite-dimensional space having strong computable type.

The space~$X$ is built as follows. Let~$f:\SS_1\to Q$ be a so-called ``space-filling curve'', i.e.~a surjective continuous function (such a function from~$[0,1]$ to~$Q$ exists by the Hahn-Mazurkiewicz theorem, see Theorem 6.3.14 in \cite{Eng89}, and can be extended to the circle by joining the two endpoints). Let~$X=Q\cup_f \B_2$ be the adjunction space, which is the quotient of the disjoint union~$Q\sqcup\B_2$ by the equivalence relation generated by~$x\sim f(x)$ for~$x\in \SS_1=\partial \B_2$.

Note that~$X$ contains a proper subset which is homeomorphic to~$X$: indeed,~$X$ properly contains~$Q$, which contains a copy of~$X$.
\begin{prop}
If~$f:\SS_1\to Q$ is surjective continuous, then the space~$Q\cup_f\B_2$ has strong computable type.
\end{prop}
\begin{proof}
Although~$X$ cannot be~$\H_2$-minimal, we show that there is a continuous function~$q:X\to\SS_2$ which is not null-homotopic (i.e.~$X\in\H_2$) and such that for any compact subset~$Y\subsetneq X$, the restriction~$q|_Y:Y\to\SS_2$ is null-homotopic.

The condition that~$f$ is surjective implies that~$Q$ has empty interior in~$X$: every~$y\in Q$ is the image by~$f$ of some~$x\in\partial \B_2$, which is a limit of a sequence~$x_n\in\B_2\setminus\partial\B_2$; therefore in~$X$,~$y$ is the limit of~$x_n\in X\setminus Q$.

The quotient space~$X/Q$ is homeomorphic to~$\B_2/\partial \SS_2$, which is homeomorphic to~$\SS_2$, so we can type the quotient map~$q:X\to X/Q$ as~$q:X\to \SS_2$. 
We claim that~$q$ is not null-homotopic.

Proposition 0.17 in \cite{Hatcher02} states that if a pair~$(X,A)$ satisfies the homotopy extension property and~$A$ is contractible, then the quotient map~$q:X\to X/A$ is a homotopy equivalence. The pair~$(X,Q)$ satisfies the homotopy extension property because~$X$ is obtained by attaching a disk to~$Q$ along its boundary (see the proof of Proposition 0.16 in \cite{Hatcher02}). The space~$Q$ is contractible, so the quotient map~$q:X\to X/Q\cong \SS_2$ is a homotopy equivalence, i.e.~there exists a function~$p:\SS_2\to X$ such that~$p\circ q$ is homotopic to~$\id_X$ and~$q\circ p$ is homotopic to~$\id_{\SS_2}$. It implies that~$q$ is not null-homotopic: otherwise the function~$\id_{\SS_2}$, which is homotopic to~$q\circ p$, would be null-homotopic as well, i.e.~$\SS_2$ would be contractible.

We now show that if~$Y$ is a compact proper subset of~$X$, then the restriction~$q|_Y:Y\to\SS_2$ is null-homotopic. As~$Q$ has empty interior in~$X$, the interior of~$\B_2$ is dense in~$X$ so it contains a point~$x\in X\setminus Y$. The space~$X\setminus \{x\}$ is contractible, because it deformation retracts to~$Q$ which is contractible. Therefore, the restriction of~$q$ to~$X\setminus \{x\}$ is null-homotopic, so its restriction to~$Y\subseteq X\setminus\{x\}$ is null-homotopic, by restriction of the homotopy.

As a result, for a compact set~$Y\subseteq X$, one has~$Y\neq X$ if and only if the restriction of~$q$ to~$Y$ is null-homotopic.

We now prove that~$X$ has computable type. Let~$X_0\subseteq Q$ be a semicomputable copy of~$X$, and~$\phi:X_0\to X$ a homeomorphism. The previous discussion implies that for an open set~$U\subseteq Q$, $U$ intersects~$X_0$ if and only if the restriction of~$q\circ \phi$ to~$X_0\setminus U$ is null-homotopic.

We now use the properties of the numberings~$\nu_{[X;Y]}$ from Theorem \ref{thm_numbering}. Let~$i_0$ be an index of a constant computable function~$f_{i_0}:Q\to\SS_2$. Note that~$i_0\in\dom(\nu_{[Z;\SS_2]})$ for all~$Z\subseteq Q$. Let~$i_1$ be such that~$\nu_{[X_0;\SS_2]}(i_1)$ is the homotopy class of~$q\circ\phi$. The definition of~$\nu_{[X;Y]}$ implies that for any open set~$U\subseteq Q$,~$i_1\in\dom(\nu_{[X_0\setminus U;\SS_2]})$ and~$\nu_{[X_0\setminus U;\SS_2]}(i_1)$ is the homotopy class of the restriction of~$q$ to~$X_0\setminus U$.

We saw that a basic open set~$U\subseteq Q$ intersects~$X_0$ if and only if~$q|_{X_0\setminus U}$ is null-homotopic, which is equivalent to~$\nu_{[X_0\setminus U;\SS_2]}(i_1)=\nu_{[X_0\setminus U;\SS_2]}(i_0)$. As~$X_0$ is semicomputable, so is~$X_0\setminus U$, therefore this predicate is c.e.~and~$X_0$ is computable, which shows that~$X_0$ has computable type.

The argument holds relative to any oracle, i.e.~if~$X_0$ is semicomputable relative to an oracle~$O$, then~$X_0$ is computable relative to~$O$. Therefore,~$X$ has strong computable type.
\end{proof}

\section{Open questions}
Our study leaves several questions open.

We have revisited several results from the literature on computable type, identifying a~$\Sigma^0_2$ invariant for which the space or the pair is minimal. We do not know whether the results in \cite{AH22} on finite simplicial complexes can be revisited in a similar way.
\begin{question}
\label{question_invariant} If a finite simplicial pair~$(X,A)$ has strong computable type, is it minimal satisfying some~$\Sigma^0_2$ invariant?
\end{question}

The example from Section \ref{sec_infinite} is a space that has strong computable type and properly contains copies of itself. The argument strongly relies on the fact that it contains the Hilbert cube, which is infinite-dimensional.
\begin{question}
\label{question_contain}Is there a finite-dimensional space that has strong computable type and properly contains copies of itself?
\end{question}

Theorem~\ref{thm_choice} assumes that the pair is minimal for some~$\Sigma^0_2$ invariant. Can this assumption be dropped?
\begin{question}\label{que: choice}Is it always true that if~$(X,A)$ has strong computable type, then~$\SCT_{(X,A)}\leq^t_W\C_\N$?
\end{question}


\begin{question}
\label{que: finitely many}Is there a compact space having (strong) computable type and infinitely many connected components?
\end{question}

We also mention a surprisingly difficult question, already raised in \cite{CelarI21}.

\begin{question}
\label{que: product}
If two pairs~$(X_{1},A_{1})$ and~$(X_{2},A_{2})$ have (strong) computable type, does~$(X_{1},A_{1})\times(X_{2},A_{2})=(X_1\times X_2,X_1\times A_2\cup A_1\times X_2)$ have (strong) computable type?
\end{question}

\paragraph{Acknowledgements.} We thank Emmanuel Jeandel for interesting discussions on the topic, as well as the anonymous referees for their careful reading and helpful comments.



\bibliographystyle{alpha}           
\bibliography{bibliojuly2023}        
\appendix
\section{Compact \texorpdfstring{$n$}{n}-manifolds are \texorpdfstring{$\H_n$}{Hn}-minimal}

\subsection{Background}
We gather some results from Bredon \cite{Bredon93} and Hatcher \cite{Hatcher02}. We do not recall the definitions of homology groups, but recall results from algebraic topology and then show how they imply Theorem \ref{thm_manifold_hn}.

If~$X$ is a topological space,~$A\subseteq X$ is a subset and~$G$ is an abelian group, then for each~$n\in\N$ one can define the~$n$th \textbf{homology group}~$H_n(X,A;G)$, which is an abelian group. When~$A=\emptyset$, we write~$H_n(X;G)$ for~$H_n(X,A;G)$. When~$G=\Z$, one denotes~$H_n(X,A;\Z)$ by~$H_n(X,A)$.

Homology of manifolds is very well understood. We will mainly use the fact that homology can detect compactness of manifolds, as follows.

\begin{thm}[Corollary VI.7.12 (p.~346) in \cite{Bredon93}]\label{thm_homology_manifold}
Let~$M$ be a connected~$n$-manifold and assume that~$G$ contains an element of order~$2$. One has
\begin{equation*}
H_n(M)\ncong 0\iff M\text{ is compact.}
\end{equation*}
\end{thm}

The next result is a generalization to pairs.
\begin{thm}[Corollary VI.7.11 (p.~346) in \cite{Bredon93}]\label{thm_homology_pair}
Let~$M$ be an~$n$-manifold and assume that~$G$ contains an element of order~$2$. If~$A\subseteq M$ is closed and connected, then
\begin{equation*}
H_n(M,M\setminus A;G)\ncong 0\iff A\text{ is compact.}
\end{equation*}
\end{thm}
Note that Theorem \ref{thm_homology_manifold} is just a particular case of Theorem \ref{thm_homology_pair}, with~$A=M$.

\subsection{Proof of Theorem \ref{thm_manifold_hn}}
These results imply that if~$(M,\partial M)$ is a compact connected~$n$-manifold with (possibly empty) boundary, then removing a point makes the $n$th homology group trivial, because compactness is lost. The argument is very classical but we did not find this particular statement in any reference textbook, so we include a proof. For~$x\in M$, we write~$M_x=M\setminus\{x\}$.

\begin{thm}\label{thm_homology_manifold_boundary}
Let~$(M,\partial M)$ be a compact connected~$n$-manifold with (possibly empty) boundary. If~$G$ contains an element of order~$2$, then~$H_n(M,\partial M;G)\ncong 0$. If~$x\in M\setminus \partial M$ then~$H_n(M_x,\partial M;G)\cong 0$.
\end{thm}
\begin{proof}
If~$\partial M$ is empty, then it is a direct consequence of Theorem \ref{thm_homology_manifold}, because~$M_x$ is a non-compact connected~$n$-manifold. Let us now consider the case when~$\partial M$ is non-empty.

In~$M$,~$\partial M$ has a \emph{collar neighborhood}, i.e.~an open neighborhood that is homeomorphic to~$\partial M\times [0,1)$, in which~$\partial M$ is identified with~$\partial M \times \{0\}$ (Proposition 3.42 in \cite{Hatcher02}). Moreover we can assume that this collar neighborhood does not contain~$x$ (otherwise take the subset that is sent to~$\partial M\times [0,\epsilon)$ by the homeomorphism, for sufficiently small~$\epsilon>0$). For simplicity of notation, we identify the collar neighborhood with~$\partial M\times [0,1)$. As the collar neighborhood deformation retracts to~$\partial M$, one has for any abelian group~$G$,
\begin{align*}
H_n(M_x,\partial M;G)&\cong H_n(M_x,\partial M\times [0,1);G)\\
&\cong H_n(M_x\setminus \partial M,\partial M\times (0,1);G)
\end{align*}
where the last isomorphism exists by excision (it is a classical derivation that can be found for instance in \cite{Bredon93}, at the beginning of section VI.9). Similarly,~$H_n(M,\partial M;G)\cong H_n(M\setminus \partial M,\partial M\times (0,1);G)$. We apply Theorem \ref{thm_homology_pair} to the manifolds~$N=M\setminus \partial M$ and~$N_x=M_x\setminus \partial M$ and their subsets~$A=N\setminus (\partial M\times (0,1))$ and~$A_x=N_x\setminus (\partial M\times (0,1))$. $A$ is compact but~$A_x$ is not because~$x$ has been removed, so by Theorem \ref{thm_homology_pair},
\begin{align*}
H_n(M\setminus \partial M,\partial M\times (0,1);G)&=H_n(N,N\setminus A)\ncong 0.\\
H_n(M_x\setminus \partial M,\partial M\times (0,1);G)&=H_n(N_x,N_x\setminus A_x)\cong 0.
\end{align*}
As a result,~$H_n(M,\partial M;G)\ncong 0$ and~$H_n(M_x,\partial M;G)\cong 0$.
\end{proof}

We can now prove that compact connected~$n$-manifolds with boundary are~$\H_n$-minimal, using the same argument as for manifolds without boundary.

If~$x\in M\setminus\partial M$ and~$B\subseteq M\setminus \partial M$ is an open Euclidean ball around~$x$, then~$M_x$ \emph{deformation retracts} to~$M\setminus B$, which means that there is a retraction~$r:M_x\to M\setminus B$ such that if~$i:M\setminus B\to M_x$ is the inclusion map, then~$r\circ i$ is homotopic to~$\id_{M_x}$, and the homotopy is constant on~$M\setminus B$. It is a classical result that deformation retractions preserve homology groups (Proposition 2.19 in \cite{Hatcher02}), so~$H_n(M\setminus B,\partial M;G)\cong H_n(M_x,\partial M;G)\cong 0$.

Theorem \ref{thm_hn_homology} relates~$\H_n$ with homology groups, but is only for single sets. However, pairs can be reduced to single sets as follows. One the one hand, relative homology groups of pairs satisfying the homotopy extension property are isomorphic to homology groups of their quotients, so:
\begin{align*}
H_n(M/\partial M;G)&\cong H_n(M,\partial M;G)\ncong 0,\\
H_n(M_x/\partial M;G)&\cong H_n(M_x,\partial M;G)\cong 0.
\end{align*}
On the other hand, a pair~$(X,A)$ is in~$\H_n$ iff~$X/A$ is in~$\H_n$ (Proposition \ref{prop_quotient}).

Therefore, Theorem \ref{thm_hn_homology} implies that~$(M,\partial M)\in\H_n$ and~$(M\setminus B,\partial M)\notin\H_n$. It implies that~$(M,\partial M)$ is~$\H_n$-minimal by Lemma \ref{lem_hn_inclusion}: if~$(X,A)$ is a proper subspair of~$M$, then~$(X,A)\subseteq (M\setminus B,\partial M)$ for some~$B\subseteq M\setminus \partial M$, so~$(X,A)\notin \H_n$.

\end{document}